\documentclass[11pt,a4paper]{article}

\usepackage{epsf,epsfig,amsfonts,amsgen,amsmath,amstext,amsbsy,amsopn,amsthm}
\usepackage{amsmath,times,mathptmx}
\usepackage{amsfonts,amsthm,amssymb}
\usepackage{graphicx}
\usepackage{latexsym,bm}
\usepackage{indentfirst}
\usepackage{color}
\usepackage[colorlinks=true,anchorcolor=blue,filecolor=blue,linkcolor=blue,urlcolor=blue,citecolor=blue]{hyperref}
\usepackage{float}
\usepackage{tikz}
\usepackage{verbatim}
\usepackage{mathrsfs}

\evensidemargin=\oddsidemargin
\usepackage[margin=2.5cm]{geometry}

\newtheorem{thm}{Theorem}[section]
\newtheorem{prob}[thm]{Problem}

\newtheorem{lem}[thm]{Lemma}
\newtheorem{cor}[thm]{Corollary}

\newtheorem{claim}{Claim}[section]
\newtheorem{definition}{Definition}[section]

\addtocounter{section}{0}

\usepackage{enumitem}  

\begin{document}
\title{\bf Edge-spectral supersaturation and stability: Resolving a problem on color-critical subgraphs
\footnote{Supported by the National Natural Science Foundation of China (Nos.\,12501471,\,12271162,\,12571369),
and the Natural Science Foundation of Shanghai (No.\,22ZR1416300).}}
\author{{\bf Longfei Fang$^{a}$},~{\bf Huiqiu Lin$^{b}$},~
{\bf Mingqing Zhai$^{c}$}\thanks{Corresponding author: mqzhai@njust.edu.cn
(M. Zhai)} \\[2mm]
\small $^{a}$ School of Mathematics and Finance, Chuzhou University,\\
\small  Chuzhou, Anhui 239012, China\\
\small $^{b}$ School of Mathematics, East China University of Science and Technology, \\
\small  Shanghai 200237, China\\
\small $^{c}$ School of Mathematics and Statistics, Nanjing University of Science and
Technology, \\
\small Nanjing, Jiangsu 210094, China
}

\date{}
\maketitle

\begin{abstract}

Inspired by Mubayi's work on the supersaturation of color-critical subgraphs,
we investigate its edge-spectral counterpart:
determining the minimum number of copies of a color-critical subgraph
in a graph whose spectral radius exceeds the corresponding extremal threshold.
While earlier breakthroughs by
Bollob\'{a}s and Nikiforov,
Ning and Zhai, and Li, Liu, and Zhang provided edge-spectral supersaturation results for the complete graph,
the general problem for color-critical graphs has remained open.
In this paper,
we strengthen existing results by studying all color-critical graphs \(F\) with \(\chi(F)=r+1\geq 4\).
Specifically,
we prove that if $\rho(G)\!>\!\sqrt{(1\!-\!1/r)2m}$,
then the number of copies of \(F\) in \(G\) satisfies
$N_F(G) \geq \bigl(\gamma_F-o(1)\bigr)m^{(|F|-2)/2},$
where both the order of magnitude and the leading constant \(\gamma_F\) are optimal.
This provides an asymptotically sharp answer, for color-critical graphs with chromatic number at least four, to a problem raised by Li, Liu, and Zhang [J. Combin.\,Theory, Ser.\,B, 2026].
In particular, we obtain the sharp bound $$N_{K_{r+1}}(G)
\geq \Big(\big(\frac{2}{r(r\!-\!1)}\big)^{(r\!-\!1)/2}\!-\!o(1)\Big)m^{(r\!-\!1)/2},$$
refining the result $N_{K_{r+1}}(G)=\Omega_r(m^{(r\!-\!1)/2})$ due to Li, Liu, and Zhang.
Our proof develops a novel framework based on the concept of
$\varepsilon$-dense subgraphs and analyzes the structures obtained via an iterative search procedure;
this approach may be of independent interest.

The above supersaturation statement is derived from a more general edge-spectral supersaturation and stability framework.
Classically, a theorem of Erd\H{o}s and Simonovits [Combinatorica, 1983]
asserts that any $n$-vertex graph with edge density strictly exceeding the Tur\'{a}n density of $F$ must contain $\Omega(n^{|F|})$ copies of $F$.
By applying the progressive induction method to spectral extremal problems,
we establish a comprehensive edge-spectral analogue of the Erd\H{o}s--Simonovits theorem.
Beyond obtaining the asymptotic upper bounds on the spectral radius $\rho(G)$ under the condition $N_F(G) = o(m^{|F|/2})$,
we further establish a sharp stability result.
Our work strengthens both the classical Erd\H{o}s--Stone--Simonovits theorem
and Nikiforov's spectral stability theorem
by extending the setting from $F$-free graphs to graphs
containing only a limited number of copies of $F$.
We conclude with several open problems in spectral supersaturation.
\end{abstract}

\begin{flushleft}
\end{flushleft}

\section{Introduction}

A central problem in extremal graph theory is the supersaturation problem, which asks for the minimum number of copies of a graph $F$
that must occur in an $n$-vertex graph whose number of edges exceeds the Tur\'{a}n threshold.
A foundational result along this line is due to Erd\H{o}s and Simonovits \cite{ES1983}.

\begin{thm}[Erd\H{o}s--Simonovits \cite{ES1983}]
\label{lem-ES-super}
Let $F$ be a graph on $f$ vertices with chromatic number $\chi (F)=r+1$. Then for any $\varepsilon >0$,
there exist $n_0$ and $\delta>0$ such that if $G$ is a graph on $n\ge n_0$ vertices with $e(G)\ge (1-\frac{1}{r} + \varepsilon )\frac{n^2}{2}$,
then $G$ contains at least $\delta n^f$ copies of $F$.
\end{thm}

Let $T_{n,r}$ denote the $n$-vertex Tur\'{a}n graph,
i.e., a complete $r$-partite graph where all partition classes are nearly equal in size.
Tur\'{a}n's celebrated theorem states that any $n$-vertex graph with $e(T_{n,r})+1$ edges contains at least one copy of $K_{r+1}$.
A natural refinement asks for the minimum number of such copies.
The study of such supersaturation problems originated with a 1941 result of Rademacher,
who proved that any $n$-vertex graph with $e(T_{n,2}) + 1$ edges
 contains at least $\lfloor n/2 \rfloor$ copies of $K_3$.
This line of inquiry was significantly extended to more general graphs $F$.
Recall that a graph $F$ with $\chi(F) = r+1$ is \emph{color-critical}
if it contains an edge whose deletion reduces its chromatic number.
For such graphs, Simonovits \cite{Simonovits1968} established that $T_{n,r}$ is the unique extremal graph for sufficiently large $n$;
that is, any $n$-vertex graph with $e(T_{n,r})+1$ edges must contain at least one copy of $F$.
Mubayi \cite{Mubayi2010} extended Simonovits' theorem to the following counting version by employing an innovative and
unified approach for color-critical graphs.
To state these results precisely, let $N_F(G)$ denote the number of copies of $F$ in a graph $G$,
and let $c(n, F)$ be the minimum number of copies of $F$ in a graph obtained by adding a single edge to $T_{n,r}$.

\begin{thm}\emph{(Mubayi \cite{Mubayi2010})}\label{thm1.1C}
Let $F$ be a color-critical graph
with $\chi (F)=r+1\geq3$.
Then, there exists $\delta =\delta_F>0$ such that
for all sufficiently large $n$, any positive integer $q<\delta n$, and
every $n$-vertex graph $G$ with
$e(G)\ge e(T_{n,r})+q,$ we have $N_F(G)\geq q\cdot c(n,F)$.
\end{thm}

Since then, the study of supersaturation problems has become very popular.
For further details on this topic, we refer the reader to the references \cite{LS1983,MY2025,Nikiforov2011,PY2017}.

Classical extremal graph theory is profoundly enriched by the study of supersaturation.
Complementing these combinatorial investigations is spectral extremal graph theory,
a vibrant field at the intersection of linear algebra
and combinatorics that characterizes the structural properties of graphs through the eigenvalues of their associated matrices,
most notably the adjacency spectral radius $\rho(G)$.
A central theme in this area is the spectral Tur\'{a}n problem, which seeks to determine the maximum spectral radius of an $F$-free graph.
While classical Tur\'{a}n-type results are traditionally formulated in terms of the number of vertices $n$,
the \textit{edge-spectral} paradigm has gained significant attention by characterizing the spectral radius relative to the number of edges $m$
(see, e.g., \cite{LZS2024, Li2025+, NZ2023}),
offering more refined structural insights across various density regimes.

A more profound inquiry naturally arises:
what structural properties are inherited by the spectral radius when a graph contains a limited number of copies of $F$,
rather than being strictly $F$-free?
This question explores the interplay between spectral extremal theory and the supersaturation phenomenon.
In the spirit of the classical Erd\H{o}s--Simonovits theorem,
which asserts that an edge density slightly exceeding the Tur\'{a}n threshold necessitates a significant count of $F$-subgraphs,
we investigate whether a spectral radius surpassing a critical threshold (suitably scaled by $\sqrt{m}$) similarly forces a proliferation of $F$-copies.
To this end, we first establish an edge-spectral analogue of the Erd\H{o}s--Simonovits theorem along with a stability result.
These results provide the necessary groundwork for our subsequent analysis of edge-supersaturation problems for color-critical graphs.

\subsection{Spectral extremal values for graphs with few copies of $F$}

Given a graph $G$, let $A(G)$ denote its adjacency matrix
and $\rho(G)$ denote the spectral radius of $A(G)$.
A covering of a graph is defined as a set of vertices that intersects all edges of the graph.
Obviously, a graph $G$ admits an independent covering if and only if $G$ is bipartite.
We use $\beta'(F)$ to denote the minimum cardinality of an independent covering of $F$.
For simplicity, we usually disregard possible isolated vertices in graphs under consideration when there is no risk of confusion.
In the present work,
we initially employ the method of progressive induction to tackle spectral problems.
More precisely, we establish an edge-spectral analogue
of the Erd\H{o}s--Simonovits result,
which provides the theoretical groundwork for our subsequent results.

\begin{thm}\label{thm1.3C}
Let $F$ be a graph of order $f$ satisfying $\chi(F)=r+1\geq2$.
If $r=1$, assume that $\beta'(F)\geq2$.
For any $\varepsilon >0$,
there exists $\delta>0$ such that for every graph $G$ of sufficiently large size $m$ satisfying
\[
\rho(G)\geq
\begin{cases}
\sqrt{2\left(1-\frac{1}{r}+\varepsilon\right)m}
& \text{if } r\geq 2,\\[3mm]
\sqrt{(1+\varepsilon)m}
& \text{if } r=1 \text{ and } \beta'(F)\geq 2,
\end{cases}
\]
we have $N_F(G)\geq \delta m^{f/2}$.
\end{thm}

The Brualdi-Hoffman-Tur\'{a}n-type problem aims to determine the maximum spectral radius of any $F$-free graph with $m$ edges.
Notably, our result offers an asymptotic solution to this problem
and extends the existing framework on $F$-free graphs \cite{LZZ25,LZS2024,Li2025+,Li2025+C,LLLY26,ZLS21}
to a broader class of graphs containing a limited number of copies of $F$.

By Theorem \ref{thm1.3C}, if $G$ is an $m$-edge graph with $N_F(G)=o(m^{f/2})$,
then $\rho(G)\leq\sqrt{(1-1/r+o(1))2m}$ when $r\geq2$ and $\rho(G)\leq\sqrt{(1+o(1))m}$ when $r=1$ and $\beta'(F)\geq 2$,
where the coefficient of the leading term $\sqrt{m}$ is optimal.
Indeed, when $r=1$ and $\beta'(F)\geq 2$, this bound is achieved by $K_{1,m}$.
For $r\geq 2$, we can select $n$ to be an integer such that $e(T_{n,r})\leq m<e(T_{n+1,r})$,
and choose a graph $G^\star$ with $m$ edges such that $T_{n,r}\subseteq G^\star \subsetneq T_{n+1,r}$.
Clearly, $G^\star$ is $F$-free since $\chi(F)>r$.
Next, we lower bound $\rho(G^\star)$.

It is known that $e(T_{n,r})\leq (1-\frac1r)\frac {n^2}2$.
Given that $m<e(T_{n+1,r})$,
it follows that $(n+1)^2\geq 2m/(1\!-\!\frac1r)$.
Observe that $G^\star$ contains the $(r\!-\!1)\lfloor\frac{n}{r}\rfloor$-regular graph $T_{r\lfloor \frac{n}{r}\rfloor,r}$
as a subgraph. Consequently,
$$\rho(G^\star)\geq \rho(T_{r\lfloor \frac{n}{r}\rfloor,r})
=(r\!-\!1)\big\lfloor\frac{n}{r}\big\rfloor\geq\sqrt{\big(1\!-\!\frac{1}{r}\!-\!o(1)\big)2m}.$$
Therefore, for every $r\geq 2$, the coefficient $\sqrt{2(1\!-\!\frac{1}{r})}$ of $\sqrt{m}$ is optimal.

\subsection{An edge-spectral stability result for graphs with few copies of $F$}

We start with two definitions that will be used to formulate our stability statements.
Given two graphs $G$ and $H$ (which may have distinct vertex sets),
we define their \emph{distance} as
$$d(G, H):=|E(G)\setminus E(H)|+|E(H)\setminus E(G)|,$$
which counts the number of edges belonging to exactly
one of $G$ and $H$.
For disjoint vertex sets $V_1,\dots,V_r$,
we use $K_{V_1,\dots,V_r}$ to represent the complete $r$-partite graph with parts $V_1,\dots,V_r$.
Building upon the spectral-extremal stability theorem established by Li, Liu and Zhang \cite{Li2025+},
we derive the following stability result for edge-spectral supersaturation via a different approach.

\begin{thm}[Edge-spectral supersaturation-stability]\label{thm1.6C}
Let $F$ be a fixed graph of order $f$ with
$\chi(F)=r+1\geq 2$.
If $r=1$, assume that $\beta'(F)\geq 2$.
For every $\varepsilon>0$, there exist constants
$\eta,\delta>0$ and $m_0$
such that, for every graph $G$ with $m=e(G)\geq m_0$ and
$N_F(G)\leq \eta m^{f/2}$,
the following assertions hold:

{\rm (i)}
If $r\geq 3$ and $\rho(G)\geq\sqrt{(1-\frac{1}{r}-\delta)2m}$,
then there exists a Tur\'{a}n graph $T_{n,r}$ such that $V(T_{n,r})\subseteq V(G)$ and $d(G,T_{n,r})\leq \varepsilon m$;

{\rm (ii)} Suppose that either $r=2$, or $r=1$ and $\beta'(F)\geq 2$.
If $\rho(G)\geq \sqrt{(1-\delta)m}$,
then there exist two disjoint subsets $U,V\subseteq V(G)$ such that
$d(G,K_{U,V})\leq \varepsilon m$.
\end{thm}

From Theorem \ref{thm1.6C},
we can derive the following supersaturation stability result.
This  result generalizes the classical stability theorem of Erd\H{o}s \cite{Erdos-1967,Erdos-1968}
and Simonovits \cite{Simonovits1968},
extending its scope from the $F$-free condition to the more general setting of $N_F(G)=o(n^{f})$.

\begin{thm}[Supersaturation-stability]
\label{thm1.5C}
Let $F$ be a fixed graph of order $f$ with
$\chi(F)=r+1\geq 3$.
For every $\varepsilon>0$, there exist constants
$\eta,\delta>0$ and $n_0$
such that the following holds.
If $G$ is a graph on $n\geq n_0$ vertices satisfying
$N_F(G)\leq \eta n^f$
and
$e(G)\geq(1-\frac{1}{r}-\delta)\frac{n^2}{2}$,
then there exists a Tur\'{a}n graph $T_{n,r}$ with
$V(T_{n,r})=V(G)$ such that
$d(G,T_{n,r})\leq \varepsilon n^2$.
\end{thm}

From Theorem \ref{thm1.6C},
we can also derive the following vertex-spectral supersaturation stability result,
originally established by Fang, Li, Lin, and Ma \cite{Fang2025+}, which generalizes
Nikiforov's spectral stability theorem in the $F$-free setting \cite{Nikiforov2009}.

\begin{thm}\emph{(Vertex-spectral supersaturation-stability \cite{Fang2025+})}\label{thm1.4C}
Let $F$ be a graph on $f$ vertices with $\chi (F)=r+1\geq 3$.
For every $\varepsilon >0$,
there exist $\eta >0, \delta >0$ and $n_0$ such that if $G$
is a graph on $n\ge n_0$ vertices with
$N_F(G)\leq \eta n^{f}$ and
$\rho(G) \ge (1-\frac{1}{r} -\delta )n$,
then $G$ can be obtained from the $r$-partite Tur\'{a}n graph $T_{n,r}$ by adding and deleting at most $\varepsilon n^2$ edges.
\end{thm}

In contrast to Theorems \ref{thm1.5C} and \ref{thm1.4C},
which apply only to dense graphs,
Theorem \ref{thm1.6C} extends to graphs of arbitrary edge densities,
including both dense and sparse cases.
To highlight this extended scope, consider the star \(K_{1,m}\),
a \(C_{\ell}\)-free graph with \(m+1\) vertices and \(m\) edges.
With spectral radius \(\rho(K_{1,m})=\sqrt{m}\) and
the edge-to-vertex ratio $e(K_{1,m})/|K_{1,m}|<1$,
it serves as a clear example of a sparse graph covered by Theorem \ref{thm1.6C} but not by the above two theorems.

\subsection{Edge-spectral supersaturation for color-critical graphs}

The study of counting problems related to spectral conditions was initiated by Bollob\'{a}s and Nikiforov \cite{Bollobas2007},
who established the inequality
\begin{align*}
\rho^{r+1}(G)\leq (r+1)N_{K_{r+1}}(G)+\sum_{s=2}^{r}(s-1)N_{K_s}(G)\rho^{r+1-s}(G)
\end{align*}
for $r\geq 2$.
In particular, for $r=2$, this yields
$N_{K_3}(G)\geq \frac{1}{3}\rho(G)(\rho^2(G) - m)$.
Ning and Zhai \cite{NZ2023} subsequently demonstrated that if $G$ is an $m$-edge graph with $\rho(G)\geq \sqrt{m}$,
then \(N_{K_3}(G)\geq \lfloor \frac{\sqrt{m}-1}{2}\rfloor\), unless $G$ is a complete bipartite graph (possibly with isolated vertices), and this bound is the best possible.
Recently, Li, Liu, and Zhang \cite{Li2025+B}
extended this line of inquiry by showing that,
for every fixed integer $r\geq 2$,
$N_{K_{r+1}}(G)=\Omega_r(m^{\frac{r-1}{2}})$
whenever $G$ is an $m$-edge graph satisfying
$\rho^2(G)>(1-\frac{1}{r})2m$.
Furthermore, Li, Liu, and Zhang \cite{Li2025+B} proposed
investigating spectral supersaturation for general color-critical graphs.
Building on their research outlook, we formalize this research direction into the following problem.

\begin{prob}[\cite{Li2025+B}]\label{ques0}
Investigate spectral supersaturation
for general color-critical graphs with chromatic number $r+1$ in an $m$-edge graph with spectral radius greater than $\sqrt{\big(1-\frac{1}{r}\big)2m}$.
\end{prob}

Let $F$ be a color-critical graph of order $f$ with chromatic number $\chi(F)=r+1$.
Recall that $c(n,F)$ denotes the minimum number of copies of $F$ among all graphs obtained by adding a single edge to $T_{n,r}$.
By Lemma \ref{lem2.2C} (to be presented later), we observe that $c(n, F)=\Theta(n^{f-2})$,
with $\alpha_F$ being precisely the coefficient of the leading term $n^{f-2}$ in $c(n, F)$.
As our next main finding, we obtain an asymptotically tight answer to Problem \ref{ques0} under the condition $\chi(F)=r+1\geq 4$.

\begin{thm}\label{thm1.8C}
Let $F$ be a color-critical graph of order $f$ with $\chi(F)=r+1\geq 4$.
For every $\varepsilon>0$, there exists $m_0$ such that for every
$m\ge m_0$, if $G$ is an $m$-edge graph with no isolated vertices
satisfying $\rho(G)\geq \sqrt{(1-{1}/{r})2m}$,
then
$$N_F(G)\geq \Big(\big(\frac{2r}{r-1}\big)^{\frac{f-2}{2}}\alpha_F\!-\!\varepsilon\Big)m^{\frac{f-2}{2}},$$
unless $G$ is a regular complete $r$-partite graph.
Furthermore, this bound is asymptotically optimal:
both the order $m^{\frac{f-2}{2}}$ and the leading constant are best possible.
\end{thm}

From Theorem \ref{thm1.8C}, we can directly derive the following spectral extremal result for color-critical graphs established by Li, Liu, and Zhang \cite{Li2025+C}.

\begin{cor}[\cite{Li2025+C}]\label{cor1.9C}
Let $F$ be a color-critical graph of order $f$ with  $\chi(F)=r+1\geq 4$.
For sufficiently large $m$,
if $G$ is an $F$-free graph with $m$ edges and no isolated vertices, then we have
 \begin{align*}
\rho(G)\leq \sqrt{\big(1\!-\!\frac{1}{r}\big)2m}.
\end{align*}
Equality holds if and only if $G$ is a regular complete \(r\)-partite graph.
\end{cor}

A straightforward calculation yields $c(n,K_{r+1})=(1+o(1))(\frac nr)^{r-1}$.
Since $\alpha_F$ is the coefficient of the leading term in $c(n,F)$,
it follows that $\alpha_{K_{r+1}}=r^{-(r-1)}$.
By applying Theorem \ref{thm1.8C},
we refine the lower bound for \(N_{K_{r+1}}(G)\) established by Li, Liu, and Zhang \cite{Li2025+B},
improving it from \(\Omega_r(m^{\frac{r-1}{2}})\) to an asymptotically tight value.

\begin{cor}
For $r\geq 3$ and sufficiently large $m$,
if $G$ is a graph with $m$ edges and no isolated vertices such that $\rho(G)\geq \sqrt{\big(1-\frac{1}{r}\big)2m},$ then we have
$$N_{K_{r+1}}(G)
\geq \Big(\big(\frac{2}{r(r\!-\!1)}\big)^{\frac{r\!-\!1}{2}}\!-\!o(1)\Big)m^{\frac{r\!-\!1}{2}},$$
unless $G$ is a regular complete $r$-partite graph.
Furthermore, this bound is asymptotically tight.
\end{cor}

\section{Preliminaries}

\subsection{Estimating the value of $c(n,F)$}

Let $F$ be a fixed color-critical graph with
$\chi(F)=r+1\geq 4$.
In this context,
we state the exact formula for \(c(n,F)\) originally derived by Mubayi \cite{Mubayi2010},
and give a detailed, self‑contained proof for the sake of completeness and for clarity in subsequent applications.

By the definition of $c(n,F)$,
there exists a graph $T_{n,r}^*$ obtained by adding a single extra edge
$uv$ within one partition class of the Tur\'{a}n graph $T_{n,r}$ such that
$N_F(T_{n,r}^*)=c(n,F)$.
Let
$T_{n,r}=K_{V_1,\ldots,V_r}$, where $n_i:=|V_i|$ for each $i\in[r]$
and the partition classes are labeled so that $uv\in E(T_{n,r}^*[V_1])$.
Since $T_{n,r}$ is a Tur\'{a}n graph, we have
$n_i\in\{\lfloor\frac{n}{r}\rfloor,\lceil\frac{n}{r}\rceil\}$
for every $i\in[r]$.
In particular,
$n_i=\frac{n}{r}+O(1)$.

An \emph{edge-preserving injection} from a graph
\(F=(V(F),E(F))\) to a graph \(G=(V(G),E(G))\)
is an injective map
\(\varphi:V(F)\to V(G)\) such that
\(\varphi(x)\varphi(y)\in E(G)\) for every \(xy\in E(F)\).
Consequently,
$$N_F(T_{n,r}^*)
=\frac{\#\bigl\{\text{edge-preserving injections }V(F)\to V(T_{n,r}^*)\bigr\}}{|\operatorname{Aut}(F)|}.$$

An edge $xy\in E(F)$ is called \emph{good} if
$\chi(F\!-\!xy)=r$.
For every good edge $xy$, every proper $r$-coloring of
$F\!-\!xy$ assigns the same color to $x$ and $y$.
Fix such a coloring $\chi_{xy}$ in which both $x$ and $y$
receive color $1$, and for each $i\in[r]$ define
$\tau_{xy}^{(i)}:=|\{z\in V(F)\setminus\{x,y\}:\chi_{xy}(z)=i\}|.$
For a real number $a$ and a non-negative integer $k$,
we use the falling factorial notation
\begin{align}\label{align-00H}
(a)_k:=a(a-1)\cdots(a-k+1),
\quad (a)_0:=1.
\end{align}

Every edge-preserving injection
$\varphi:V(F)\to V(T_{n,r}^*)$
must use the unique extra edge $uv$, since otherwise $\varphi$
would give an embedding of $F$ into the $r$-partite graph
$T_{n,r}$, contradicting $\chi(F)=r+1$.
Let $xy\in E(F)$ be the unique edge satisfying
$\{\varphi(x),\varphi(y)\}=\{u,v\}$.
Then $F-xy$ admits a proper $r$-coloring induced by the
partition $V_1,\ldots,V_r$ of $T_{n,r}$, and hence
$\chi(F-xy)\le r$.
On the other hand, deleting a single edge decreases the chromatic
number by at most one, so
$\chi(F-xy)\ge \chi(F)-1=r$.
Therefore,
$\chi(F-xy)=r$,
and hence $xy$ is a good edge.

Moreover, the injection $\varphi$ uniquely determines a proper
$r$-coloring of $F-xy$: for each $z\in V(F)$, define
$\chi_{\varphi}(z)=i$ if and only if
$\varphi(z)\in V_i$.
Since $\varphi(x),\varphi(y)\in V_1$, we have
$\chi_{\varphi}(x)=\chi_{\varphi}(y)=1$.
Thus, every edge-preserving injection uniquely determines a good
edge $xy$ together with a proper $r$-coloring of $F-xy$ in which
$x$ and $y$ both receive color $1$.

For a fixed good edge $xy$ and a fixed proper $r$-coloring
$\chi_{xy}$ as above, there are two ways to map
$\{x,y\}$ onto $\{u,v\}$.
The remaining $\tau_{xy}^{(1)}$ vertices of color $1$
can be mapped injectively into $V_1\setminus\{u,v\}$ in
$(n_1-2)_{\tau_{xy}^{(1)}}$
ways, while for each $i\in\{2,\ldots,r\}$, the
$\tau_{xy}^{(i)}$ vertices of color $i$ can be mapped injectively
into $V_i$ in $(n_i)_{\tau_{xy}^{(i)}}$ ways.
Therefore,
\begin{align}\label{align-00G}
\operatorname{Cont}_{T_{n,r}^*}(\chi_{xy})
=
2(n_1-2)_{\tau_{xy}^{(1)}}
\prod_{i=2}^{r}(n_i)_{\tau_{xy}^{(i)}}.
\end{align}
Conversely, every choice counted in \eqref{align-00G} gives an
edge-preserving injection from $F$ to $T_{n,r}^*$.
Indeed, the edge $xy$ is mapped onto the added edge $uv$, while
every edge of $F-xy$ has endpoints of distinct colors and is
therefore mapped to an edge of the complete $r$-partite graph
$T_{n,r}$.
Together with the uniqueness established above, this shows that
every edge-preserving injection is counted exactly once.
Summing over all good edges $xy$ and all proper $r$-colorings
$\chi_{xy}$ of $F-xy$ in which both $x$ and $y$ receive color $1$,
and then dividing by $|\operatorname{Aut}(F)|$, we obtain
\begin{align}\label{align-01F}
c(n,F)
=
\frac{1}{|\operatorname{Aut}(F)|}
\sum_{\substack{xy\in E(F)\\ xy\text{ is good}}}
\,\sum_{\substack{\chi_{xy}\text{ proper}\\
\chi_{xy}(x)=\chi_{xy}(y)=1}}
\!\!\!\!\!\!2(n_1-2)_{\tau_{xy}^{(1)}}
\prod_{i=2}^{r}(n_i)_{\tau_{xy}^{(i)}}.
\end{align}

Furthermore, Mubayi \cite{Mubayi2010} provided a useful estimate for $c(n,F)$ as follows:

\begin{lem}[Mubayi \cite{Mubayi2010}]\label{lem2.2C}
Given a color-critical graph $F$ of order $f$ with $\chi (F)=r+1\geq3$,
there exist constants $\alpha_F>0$ and $\beta_F>0$ such that for sufficiently large $n$,
$$\big|c(n, F)-\alpha_F n^{f-2}\big| < \beta_F n^{f-3}.$$
In particular, we have $\frac{1}{2}\alpha_F  n^{f-2} < c(n, F) < 2 \alpha_F n^{f-2}$.
\end{lem}

\subsection{Some stability results for $F$-free graphs}

We use the following classical Graph Removal Lemma, which can be found, for instance, in Koml\'{o}s and Simonovits \cite[Theorem 2.9]{Komlos1996}.

\begin{lem}[Graph Removal Lemma \cite{Komlos1996}]\label{lem2.1C}
Let $F$ be a fixed graph on $f$ vertices. For every $\varepsilon>0$,
there exist constants $\eta>0$ and $n_0\in\mathbb{N}$ such that
every graph $G$ on $n\geq n_0$ vertices satisfying
$N_F(G)\leq \eta n^f$
can be made $F$-free by deleting at most $\varepsilon n^2$ edges.
\end{lem}

\begin{lem}[Erd\H{o}s, Frankl, and R\"odl \cite{EFR1986}]\label{lem2.3C}
Let \(F\) be a fixed graph with chromatic number $r+1$. For any \(\varepsilon>0\),
there exists an integer \(n_0\) such that if \(G\) is an \(F\)-free graph on \(n\) vertices with \(n\ge n_0\),
then at most \(\varepsilon n^2\) edges can be removed from \(G\) to make it \(K_{r+1}\)-free.
\end{lem}

The following lemma provides a vertex-spectral stability result,
which can be derived from a nice work of Nikiforov \cite{Nikiforov2009}.

\begin{lem}[Nikiforov \cite{Nikiforov2009}] \label{lem2.4C}
Let $F$ be a graph with $\chi(F)=r\!+\!1\geq 3$.
For any $\varepsilon>0$, there exist $\delta>0$ and an integer $n_0$ such that
if $G$ is an $F$-free graph of order $n\geq n_0$ with $\rho(G)\geq(1\!-\frac{1}{r}\!-\delta)n$,
then there exists a Tur\'{a}n graph $T_{n,r}$ such that $V(T_{n,r})=V(G)$ and $d(G,T_{n,r})\leq \varepsilon n^2$.
\end{lem}

The following edge-spectral extremal result for $K_{r+1}$ is given by Nikiforov.

\begin{thm}[\cite{Nikiforov2002,Nikiforov2006,Nikiforov2009-JCTB}]\label{thm1.7C}
Let $r\geq2$ and $G$ be a $K_{r+1}$-free graph with $m$ edges and no isolated vertices.
Then
$\rho(G)\le\sqrt{\big(1-1/r\big)2m}.$
Equality holds if and only if one of the following conditions holds:

{\rm (i)}  $r= 2$ and $G$ is a complete bipartite graph;

{\rm (ii)}  $r\ge 3$ and $G$ is a regular complete $r$-partite graph.
\end{thm}

We also require the following lemmas, established by Li, Liu, and Zhang \cite{Li2025+}.

\begin{lem}[Li, Liu, and Zhang \cite{Li2025+}] \label{lem2.5C}
For any $\varepsilon>0$ and any integer $r\ge 3$, there exist $\delta>0$ and $m_{0}$ such that
if $G$ is a $K_{r+1}$-free graph of size $m\ge m_{0}$ satisfying
$\rho(G)\ge \sqrt{(1-1/r-\delta)2m},$
then $G$ admits an induced subgraph $G'$ such that
$|G'|\leq (1\!+\!\varepsilon)\sqrt{2m/(1-1/r)}$ and $d(G,G')\leq\varepsilon m$.
\end{lem}

\begin{lem}[Li, Liu, and Zhang \cite{Li2025+}] \label{lem2.6C}
 For every $\varepsilon\in (0,0.01)$,
 there exists $\delta=\delta(\varepsilon)>0$ such that if $H$
 is a triangle-free graph with $h$ edges satisfying $\rho(H)\geq\sqrt{(1-\delta)h}$,
 then there exist disjoint vertex subsets $U,V\subseteq V(H)$ such that $d(H, K_{U,V}) \leq \varepsilon h$.
\end{lem}

\begin{lem}[Li, Liu, and Zhang \cite{Li2025+}] \label{lem2.7C}
 For every $\varepsilon > 0$, there exists $\delta=\delta(\varepsilon)>0$ such that the following holds.
 Let $H$ be a graph with $h$ edges,
 and let $\boldsymbol{x}$ be the non-negative unit eigenvector corresponding to $\rho(H)$.
 If $\rho(H) \geq \sqrt{(1 - \delta) h}$ and
$\max\{x_i : i \in V(H)\} > \delta^{-1} h^{-1/4},$
then there are disjoint vertex subsets $U, V\subseteq V(H)$ such that $d(H, K_{U,V}) \leq \varepsilon h$.
\end{lem}

Let $G$ be a non-empty graph,
and let $\bm{x}=(x_v)_{v\in V(G)}$ be a non-negative unit eigenvector of $G$ corresponding to $\rho(G)$.
An edge $uv\in E(G)$ is said to be \emph{light} if $x_{u}x_{v}\leq\frac{1}{8\sqrt{e(G)}}$.

\begin{definition}\label{def1.1}
Let $G$ be a graph and let $\varepsilon\in(0,1)$
satisfy $\lfloor\varepsilon e(G)\rfloor\geq1$.
We define the following iterative procedure:

{\rm (i)} Let $G_1$ be obtained from $G$ by
removing all isolated vertices, and set $i=1$.

{\rm (ii)} If $i=\lfloor\varepsilon e(G)\rfloor$,
terminate the procedure. Otherwise, choose a
non-negative unit eigenvector $\bm{x}^{(i)}$
corresponding to $\rho(G_i)$.
If $G_i$ contains no light edges with respect
to $\bm{x}^{(i)}$, terminate the procedure.

{\rm (iii)} Choose a light edge $u_iv_i$ of
$G_i$ with respect to $\bm{x}^{(i)}$,
and let $G_{i+1}$ be obtained from
$G_i-u_iv_i$ by removing all isolated vertices.
Replace $i$ by $i+1$ and return to step {\rm (ii)}.

\par\smallskip\noindent
The procedure produces a sequence of graphs
$G_1\supset\cdots\supset G_\ell$
and a sequence of light edges
$u_1v_1,\ldots,u_{\ell-1}v_{\ell-1}$,
where $1\leq\ell\leq\lfloor\varepsilon e(G)\rfloor$.
We call the final graph $G_\ell$
an $\varepsilon$-subgraph of $G$.
\end{definition}

For every graph $H$ with $e(H)>0$, let $\Phi(H):={\rho(H)}/{\sqrt{e(H)}}$.

\begin{lem}\label{lem2.8C}
Let $G$ be a graph with $m$ edges such that
$\rho(G)\geq\sqrt{am}$, where $0.81\leq a\leq2$.
Fix $\varepsilon\in(0,1)$ such that
$\lfloor\varepsilon m\rfloor\geq1$, and let
$G_1,\ldots,G_\ell$ be the sequence defined
in Definition~\ref{def1.1} with parameter $\varepsilon$.
Then the following statements hold:

{\rm (i)} For every $i\in[\ell]$, we have
$\rho(G_i)\geq\sqrt{a\cdot e(G_i)}+\frac{i-1}{5\sqrt{m}}.$

{\rm (ii)} For every $i\in[\ell]$, we have
$\Phi(G_i)-\Phi(G_1)\geq\frac{i-1}{5m}.$

{\rm (iii)} Suppose in addition that $a>1$ and
$G$ has no isolated vertices. Let
$\bm{x}=(x_v)_{v\in V(G)}$ be a non-negative
unit eigenvector corresponding to $\rho(G)$.
If $G$ contains no light edges with respect
to $\bm{x}$, then
$$\big(\frac{a-1}{8}\big)^5m^{-1/4}<x_v<\big(\frac{8}{a-1}\big)^4m^{-1/4}$$
for every $v\in V(G)$. Furthermore,
$|G|\leq(\frac{8}{a-1})^{10}\sqrt{m}$ and
$\delta(G)\geq (\frac{a-1}{8})^9\sqrt{m}$.
\end{lem}

\begin{proof}
(i) The proof proceeds by induction.
If $i=1$, the statement holds trivially. Assume $i\geq 2$,
and let $\bm y$ be the non-negative unit
eigenvector of $G_{i-1}$ used to identify
the light edge $u_{i-1}v_{i-1}$.
Let $\widetilde G_i$ be the graph obtained
from $G_{i-1}$ by deleting only the edge
$u_{i-1}v_{i-1}$, retaining all isolated
vertices. Then
$V(\widetilde G_i)=V(G_{i-1})$
and
$\rho(\widetilde G_i)=\rho(G_i)$.
Consequently,
\begin{align*}
\rho(G_{i-1})
=\bm y^\top A(G_{i-1})\bm y
=\bm y^\top A(\widetilde G_i)\bm y
  +2y_{u_{i-1}}y_{v_{i-1}}
\leq \rho(\widetilde G_i)
  +2y_{u_{i-1}}y_{v_{i-1}}
=\rho(G_i)+2y_{u_{i-1}}y_{v_{i-1}}.
\end{align*}
Since $y_{u_{i-1}}y_{v_{i-1}}\leq\frac{1}{8\sqrt{e(G_{i-1})}}$,
we have
\begin{align}\label{align-03F}
\rho(G_{i\!-\!1})\leq \rho(G_i)\!+\!\frac{1}{4\sqrt{e(G_{i\!-\!1})}}.
\end{align}
Since $a\geq0.81$ and $e(G_{i-1})-e(G_{i})=1$,
it follows that
\begin{align}\label{align-04F}
\sqrt{a\cdot e(G_{i-1})}\!-\!\sqrt{a\cdot e(G_{i})}=\frac{\sqrt{a}}{\sqrt{e(G_{i\!-\!1})}\!+\!\sqrt{e(G_i)}}
\geq\frac{9}{20\sqrt{e(G_{i\!-\!1})}}.
\end{align}
Moreover, by the induction hypothesis, we have $\rho(G_{i-1})\geq \sqrt{a \cdot e(G_{i-1})}+\frac{i-2}{5\sqrt{m}}$.
Combining this with \eqref{align-03F} and \eqref{align-04F}, we deduce that
\begin{align*}
\rho(G_{i})
\geq \sqrt{a\cdot e(G_{i\!-\!1})}\!+\!\frac{i\!-\!2}{5\sqrt{m}}\!-\!\frac{1}{4\sqrt{e(G_{i\!-\!1})}}
\geq \sqrt{a\cdot e(G_{i})}\!+\!\frac{i\!-\!2}{5\sqrt{m}}\!+\!\frac{1}{5\sqrt{e(G_{i\!-\!1})}}.
\end{align*}
Noting that $e(G_{i-1})\leq m$,
we further derive that $\rho(G_{i})\geq \sqrt{a\cdot e(G_{i})}+\frac{i-1}{5\sqrt{m}}$,
as desired.

(ii) In light of \eqref{align-03F}, we have $\rho(G_{i+1})\geq \rho(G_{i})-\frac{1}{4\sqrt{e(G_i)}}$.
Consequently,
\begin{align*}
\rho(G_{i+1})\sqrt{e(G_{i})}\!-\!\rho(G_{i})\sqrt{e(G_{i+1})}
\geq \rho(G_{i})\big(\sqrt{e(G_{i})}\!-\!\sqrt{e(G_{i})\!-\!1}\,\big)\!-\!\frac{1}{4}
\geq \frac12\frac{\rho(G_{i})}{\sqrt{e(G_i)}}\!-\!\frac{1}{4},
\end{align*}
As indicated in (i), $\rho(G_i)\geq0.9\sqrt{e(G_i)}$, and thus $\frac12\frac{\rho(G_{i})}{\sqrt{e(G_i)}}\!-\!\frac{1}{4}
\geq \frac{1}{5}$.
It follows that
\begin{align*}
\Phi(G_{i+1})\!-\!\Phi(G_{i})
=\frac{\rho(G_{i+1})\sqrt{e(G_{i})}
\!-\!\rho(G_{i})\sqrt{e(G_{i+1})}}{\sqrt{e(G_{i+1})}\sqrt{e(G_{i})}}
\geq \frac{1}{5m}.
\end{align*}
By iterating this inequality, we can deduce that
$\Phi(G_i)-\Phi(G_1)\geq \frac{i\!-\!1}{5m}$, as required.

(iii) Set $\varepsilon:=(a-1)/8$ and let
$x_{u^*}=\max_{v\in V(G)}x_v$.
Define
$L:=\{u\in V(G):x_u\geq\varepsilon^2x_{u^*}\}$.
Since $1<a\leq2$, we have
$0<\varepsilon\leq1/8$.
We first prove that $x_{u^*}<\varepsilon^{-4} m^{-1/4}$.
Suppose that $x_{u^*}\geq \varepsilon^{-4} m^{-1/4}$.
Since $|L|\varepsilon^{4}x_{u^*}^2\leq\sum_{u\in L}x_u^2\leq\sum_{u\in V(G)}x_u^2=1$, we have $|L|\leq \varepsilon^{4}\sqrt{m}$.
Observe that for every $u\in V(G)$, there are at most $d_G(u)$ walks of length two from $u^*$ to $u$ in $G$.
From this observation, we deduce that
\begin{align*}
\rho^2(G) x_{u^*}
&\leq \!\!\!\sum_{u\in V(G)}\!\!\!d_G(u)x_u=\!\!\!\sum_{uv\in E(G)}\!\!\!(x_u\!+\!x_v)
\leq \!\!\!\sum_{uv\in E(G[L])}\!\!\!(x_u\!+\!x_v)+\!\!\!\sum_{uv\in E(G)\setminus E(G[L])}\!\!\!(x_u\!+\!x_v)\nonumber\\
&\leq 2\binom{|L|}{2}x_{u^*}+ \big(1+\varepsilon^2\big)mx_{u^*}
<\big(1+\varepsilon^2+\varepsilon^8 \big)mx_{u^*},
\end{align*}
which contradicts the assumption that $\rho^2(G)\geq am=(1+8\varepsilon)m$.
Hence, $x_{u^*}<\varepsilon^{-4} m^{-1/4}$.

Since $G$ contains no light edges,
we have $x_ux_v>\frac{1}{8}m^{-1/2}\geq \varepsilon m^{-1/2}$ for each $uv\in E(G)$.
Together with the fact that $x_{u^*}<\varepsilon^{-4}m^{-1/4}$ and the assumption that $\delta(G)\geq1$,
this implies that $\min_{v\in V(G)}x_v>\varepsilon^{5}m^{-1/4}.$
Moreover, since $\bm{x}$ is a unit eigenvector of $G$,
we have
$$\big|G\big|\cdot\big(\varepsilon^{5}m^{-1/4}\big)^2\leq \sum_{v\in V(G)}x_v^2=1.$$
Thus, we obtain the inequality $|G|\leq \varepsilon^{-10}\sqrt{m}$, as desired.

Recall that $x_{u^*}<\varepsilon^{-4}m^{-1/4}$. For any vertex $v\in V(G)$, we have
$x_v\geq \varepsilon^{5}m^{-1/4}>\varepsilon^{9}x_{u^*}$.
It follows that
$$\sqrt{m}\cdot \varepsilon^{9}x_{u^*}\leq\rho(G)x_{v}=\sum_{u\in N_G(v)}\!\!x_u\leq d_G(v)x_{u^*}.$$
This implies that $d_G(v)\geq\varepsilon^{9}\sqrt{m}$ for each $v\in V(G)$.
Thus, we have
$\delta(G)\geq \varepsilon^{9}\sqrt{m}$.
\end{proof}

\section{Proof of Theorem \ref{thm1.3C}}

The method of \emph{progressive induction}, introduced by Simonovits \cite{Simonovits1968},
is a powerful technique used to demonstrate statements that hold only for sufficiently large $n$.
It resembles mathematical induction and the Euclidean algorithm, combining elements of both in a certain sense.

\begin{lem}[Simonovits \cite{Simonovits1968}] \label{lem3.1}
Let $\mathfrak{U}=\cup_{i=1}^{\infty}\mathfrak{U}_i$ be a set of given elements,
where $\mathfrak{U}_i$ are disjoint finite subsets of $\mathfrak{U}$.
Let $B$ be a condition or property defined on $\mathfrak{U}$
(the elements of $\mathfrak{U}$ may or may not satisfy $B$).
Let $\varphi$ be a function defined on $\mathfrak{U}$ taking non-negative integer values,
and suppose that the following conditions hold:

{\rm (i)}
If $a$ satisfies $B$, then $\varphi(a)=0$;

{\rm (ii)} There exists a positive integer $M_0$ such that for all $m>M_0$ and all $a\in \mathfrak{U}_m$,
either $a$ satisfies $B$, or there exist an integer $m^*$ and an element
$a^*\in\mathfrak{U}_{m^*}$ with
$\frac{m}{2}<m^*<m$ and $\varphi(a)<\varphi(a^*)$.

\noindent
Then, there exists a positive integer $m_0$ such that for all $m>m_0$, every $a\in \mathfrak{U}_m$ satisfies the property $B$.
\end{lem}

\begin{proof}[\textbf{Proof of Theorem \ref{thm1.3C}}]

Fix $\varepsilon>0$ as in the statement of the theorem.
We divide the proof into two cases.

\vspace{1mm}
\noindent\textbf{Case 1. $\chi(F)=r+1\geq3$.}

Choose $0<\eta\leq\min\{\varepsilon,1/2,1/r\}$
sufficiently small such that
\begin{align}\label{align-05A}
2\eta^{24}\big(\frac4\eta\big)^{20}\leq \eta^3
\quad\text{and}\quad
\sqrt{2\big(1-\frac1r\big)}+\sqrt{2\eta^3}
<\sqrt{2\big(1-\frac1r+\eta\big)}.
\end{align}
Since $\eta\leq\varepsilon$, every graph
$G$ satisfying
$\rho(G)\geq\sqrt{2(1\!-\!1/r\!+\!\varepsilon)m}$
also satisfies
$\rho(G)\geq\sqrt{2(1\!-\!1/r\!+\!\eta)m}$.
Hence, it suffices to establish the theorem
with $\eta$ in place of $\varepsilon$.
Set $a:=2(1-\frac1r+\eta)$. Then $1<a\leq2$.

To apply progressive induction, we introduce
an integer-valued function in terms of $\Phi$.
For every graph $H$ with $e(H)>0$, let
$$\varphi(H):= \max\Big\{0,\Big\lfloor(\Phi(H)-\sqrt{a})\frac6\eta\Big\rfloor\Big\}.$$

We next fix the constants needed for
the edge-deletion argument.
Applying Lemma \ref{lem2.1C} with its parameter $\varepsilon$
taken to be $\eta^{24}$,
we obtain constants $\eta_{\ref{lem2.1C}}>0$ and $n_{\ref{lem2.1C}}$
such that every graph $H$ on $n\geq n_{\ref{lem2.1C}}$ vertices
with $N_F(H)<\eta_{\ref{lem2.1C}} n^f$ can be made $F$-free
by deleting at most $\eta^{24}n^2$ edges.
Similarly, by Lemma \ref{lem2.3C}, there exists
$n_{\ref{lem2.3C}}$ such that every $F$-free graph on
$n\geq n_{\ref{lem2.3C}}$ vertices can be made $K_{r+1}$-free
by deleting at most $\eta^{24}n^2$ edges.

Since $\rho(H)\leq\sqrt{2e(H)}$, we have
$0\leq\varphi(H)\leq K:=\lceil\frac{6\sqrt2}{\eta}\rceil.$
Set
$q:=(1-\eta)^{-f/2}>1$
and choose
$0<\delta<\eta_{\ref{lem2.1C}} q^{-(K+1)}$.
Throughout the proof, graphs are considered up to isomorphism.
For each positive integer $m$, let
\[
\mathfrak U_m:=
\big\{
G:e(G)=m,\ \delta(G)\geq1,\
N_F(G)<\delta q^{\varphi(G)}m^{f/2}
\big\}.
\]
Since every graph in $\mathfrak U_m$ has at most
$2m$ vertices, each $\mathfrak U_m$ is finite. Moreover, the sets
$\mathfrak U_m$ are pairwise disjoint because the edge number is fixed
to be $m$.
Let $B$ be the property
$\rho(G)<\sqrt{a e(G)}$.
If $G$ satisfies $B$, then $\varphi(G)=0$, so condition~(i) of
Lemma \ref{lem3.1} holds.

We now verify condition~(ii) of Lemma \ref{lem3.1}. Let $m$ be sufficiently large and let
$G^\star\in\mathfrak U_m$ fail $B$. Thus
$\rho(G^\star)\geq \sqrt{am}$.
Starting with $G_1:=G^\star$, apply Definition \ref{def1.1} with its parameter $\varepsilon$
taken to be $\eta$, and let $G_1\supset\cdots\supset G_\ell$ be the
resulting sequence. By Lemma \ref{lem2.8C} (i), for every
$i\in[\ell]$,
\begin{align}\label{align-05F}
\rho(G_i)\geq \sqrt{a\cdot e(G_i)}+(i-1)\big/(5\sqrt m).
\end{align}

We claim that
$\ell=\lfloor\eta m\rfloor$.
Suppose instead that $\ell<\lfloor\eta m\rfloor$. Then $G_\ell$ has
neither light edges nor isolated vertices. Set
$m':=e(G_\ell)$ and $n':=|G_\ell|$.
Since fewer than $\eta m$ edges have been removed,
$m'>(1-\eta)m$.
Hence $m'\to\infty$ as $m\to\infty$.
By \eqref{align-05F}, $\rho(G_\ell)\geq\sqrt{am'}$.

Since $r\geq 2$ and $a=2(1-\frac1r+\eta)$, we have $a-1\geq 2\eta$.
Applying Lemma \ref{lem2.8C} (iii) to $G_\ell$ with $m'$ in place of $m$, we obtain
$n'\leq (\frac{8}{a-1})^{10}\sqrt{m'}
\leq (\frac{4}{\eta})^{10}\sqrt{m'}.$
Moreover,
$n'\geq \rho(G_\ell)\geq \sqrt{am'}$, and hence $n'\to\infty$ as $m\to\infty$.
In particular, for all sufficiently large $m$,
we have
$n'\geq \max\{n_{\ref{lem2.1C}},n_{\ref{lem2.3C}}\}$,
so the order requirements of both Lemmas \ref{lem2.1C} and \ref{lem2.3C} are satisfied.

Since $G_\ell\subseteq G^\star$ and $G^\star\in\mathfrak U_m$, we have
$N_F(G_\ell)\leq N_F(G^\star)<\delta q^{\varphi(G^\star)}m^{f/2}$.
Since $\varphi(G^\star)\leq K$, it follows that
$N_F(G_\ell)<\delta q^K m^{f/2}$.
Moreover, since $m'>(1-\eta)m$ and
$q=(1-\eta)^{-f/2}$, we have
$m^{f/2}<q(m')^{f/2}$.
Hence, $N_F(G_\ell)<\delta q^{K+1}(m')^{f/2}$.
Since $\delta<\eta_{\ref{lem2.1C}}q^{-(K+1)}$, we have
$N_F(G_\ell)<\eta_{\ref{lem2.1C}} (m')^{f/2}$.
Finally, since
$m'\leq\binom{n'}{2}<(n')^2$, we obtain
$(m')^{f/2}\leq(n')^f$.
Consequently,
$N_F(G_\ell)<\eta_{\ref{lem2.1C}} (n')^f$.
Hence, by Lemma \ref{lem2.1C}, we can delete
at most $\eta^{24}(n')^2$ edges from $G_\ell$ to obtain an $F$-free
spanning subgraph $G''$. Since $n'\geq n_{\ref{lem2.3C}}$,
Lemma \ref{lem2.3C} allows us to delete at most another
$\eta^{24}(n')^2$ edges from $G''$ to obtain a $K_{r+1}$-free spanning subgraph
$G'''$.
By \eqref{align-05A} and $n'\leq (\frac{4}{\eta})^{10}\sqrt{m'}$, we have
$e(G_\ell)-e(G''') \leq2\eta^{24}(n')^2 \leq \eta^3m'$.
Let $H_0$ be the spanning subgraph of $G_\ell$ with
$V(H_0)=V(G_\ell)$ and $E(H_0)=E(G_\ell)\setminus E(G''')$.
Then $e(H_0)\leq \eta^3m'$ and
$\rho(H_0)\leq\sqrt{2e(H_0)}\leq\sqrt{2\eta^3m'}$.
Let $\widehat{G}$ be the graph obtained from $G'''$ by deleting all
isolated vertices. Clearly, $\widehat{G}$ is $K_{r+1}$-free and has
no isolated vertices. Moreover,
$e(\widehat{G})=e(G''')$ and
$\rho(\widehat{G})=\rho(G''')$.
Hence, by Theorem~\ref{thm1.7C}, we have
$$\rho(G''')=\rho(\widehat{G})
\leq\sqrt{2\Big(1-\frac1r\Big)e(\widehat{G})}
=\sqrt{2\Big(1-\frac1r\Big)e(G''')}
\leq \sqrt{2\Big(1-\frac1r\Big)m'}.$$
Therefore, using $A(G_\ell)=A(G''')+A(H_0)$ and the Rayleigh quotient,
\[
\begin{aligned}
\rho(G_\ell)
\leq\rho(G''')+\rho(H_0)
\leq
\sqrt{2\big(1-\frac1r\big)m'}+\sqrt{2\eta^3m'}
<\sqrt{2\big(1-\frac1r+\eta\big)m'}
=\sqrt{am'},
\end{aligned}
\]
where the strict inequality follows from \eqref{align-05A}.
This contradicts \eqref{align-05F},
and the claim $\ell=\lfloor\eta m\rfloor$ follows.

Now let $G^*:=G_\ell$ and $m^*:=e(G^*)=m-\ell+1$.
For sufficiently large $m$,
$\frac m2<m^*<m$ and $m^*>(1-\eta)m$.
By Lemma \ref{lem2.8C} (ii) and $\ell=\lfloor\eta m\rfloor$,
we obtain $\Phi(G^*)-\Phi(G^\star) \geq\frac{\ell-1}{5m}>\frac\eta6$.
Recall that $G^\star\in\mathfrak U_m$ fails $B$.
Then $\Phi(G^\star)\geq \sqrt a$.
Consequently,
$\lfloor (\Phi(G^*)-\sqrt a)\frac6\eta\rfloor\geq\lfloor(\Phi(G^\star)-\sqrt a)\frac6\eta\rfloor+1$.
Since $\Phi(G^\star)\geq\sqrt a$ and
$\Phi(G^*)-\Phi(G^\star)>\eta/6$,
the definition of $\varphi$ yields
$\varphi(G^*)\geq\varphi(G^\star)+1$.
Moreover, since $m^*>(1-\eta)m$ and
$q=(1-\eta)^{-f/2}$, it follows that
$q(m^*)^{f/2}>m^{f/2}$.
As $G^*\subseteq G^\star$ and
$G^\star\in\mathfrak U_m$, we obtain
$N_F(G^*)\leq N_F(G^\star)<\delta q^{\varphi(G^\star)}m^{f/2}$.
Combining these inequalities yields
$N_F(G^*)<\delta q^{\varphi(G^\star)+1}(m^*)^{f/2}$.
Recall that $\varphi(G^*)\geq\varphi(G^\star)+1$ and $q>1$.
Then $N_F(G^*)<\delta q^{\varphi(G^*)}(m^*)^{f/2}$.
Furthermore, $G^*$ has no isolated vertices by construction.
Hence,
$G^*\in\mathfrak U_{m^*}$ and $\varphi(G^\star)<\varphi(G^*)$.
This verifies condition~(ii) of Lemma \ref{lem3.1}.

Both conditions~(i) and (ii) of Lemma \ref{lem3.1} are now satisfied.
Lemma \ref{lem3.1} therefore yields that, for all sufficiently large $m$,
every graph in $\mathfrak U_m$ obeys condition $B$.
Now, let $G$ be an arbitrary graph with $e(G)=m$ such that
$N_F(G)<\delta m^{f/2}$.
Let $\widetilde G$ be obtained from $G$ by deleting all isolated vertices.
Then $e(\widetilde G)=m$,
$\rho(\widetilde G)=\rho(G)$,
$N_F(\widetilde G)\leq N_F(G)$,
and $\delta(\widetilde G)\geq1$.
Since $q^{\varphi(\widetilde G)}\geq1$, it follows that
$N_F(\widetilde G) <\delta m^{f/2} \leq \delta q^{\varphi(\widetilde G)}m^{f/2}$.
Thus $\widetilde G\in\mathfrak U_m$, so $\widetilde G$ satisfies $B$.
Consequently,
\[
\rho(G)=\rho(\widetilde G)
<\sqrt{2\big(1-\frac1r+\eta\big)m}
\leq
\sqrt{2\big(1-\frac1r+\varepsilon\big)m}.
\]
Taking the contrapositive, every $m$-edge graph \(G\) satisfying
$\rho(G)\geq\sqrt{2(1-\frac1r+\varepsilon)m}$
obeys $N_F(G)\geq\delta m^{f/2}$.
This proves Case~1.

\vspace{1mm}
\noindent\textbf{Case 2. $\chi(F)=2$ and $\beta'(F)\geq2$.}

Choose an independent vertex cover $A\subseteq V(F)$ of minimum size and
set
$a_0:=|A|=\beta'(F)$,
$B:=V(F)\setminus A$, and
$b_0:=|B|=f-a_0$.
Since $A$ is independent and covers every edge of $F$, the set $B$ is
also independent. Hence $(A,B)$ is a bipartition of $F$, and $F$ is a
spanning subgraph of $K_{a_0,b_0}$. Moreover, $\chi(F)=2$ implies that
$F$ has at least one edge, so $b_0\geq1$.

Let $F_0$ be obtained from $K_{a_0,b_0}$ by adding one edge inside the
part $A$. Since $a_0\geq2$ and $b_0\geq1$, this added edge together with
any vertex of $B$ forms a triangle. On the other hand, $F_0$ is clearly
3-colorable. Thus
$\chi(F_0)=3$.
Also, $F$ is a spanning subgraph of $F_0$, and $|F_0|=|F|=f$.

Applying Case~1 to $F_0$ with parameter $\varepsilon/2$,
we obtain a constant $\delta_0>0$ such that every graph $G$
with sufficiently many edges $m$ and
$\rho(G)\geq\sqrt{(1+\varepsilon)m}$
satisfies $N_{F_0}(G)\geq\delta_0 m^{f/2}$.

It remains to compare the numbers of copies of $F_0$ and $F$.
Let $\mathcal P$ be the set of pairs $(Q,R)$ such that $Q$ is a copy
of $F$ in $G$, $R$ is a copy of $F_0$ in $G$, and $Q\subseteq R$.
Because $F$ is a spanning subgraph of $F_0$, every copy $R$ of $F_0$
contains at least one copy $Q$ of $F$. Therefore
$N_{F_0}(G)\leq |\mathcal P|$.
Conversely, for a fixed copy $Q$ of $F$, every copy $R$ of $F_0$
containing $Q$ has the same $f$-vertex set as $Q$.
As each such copy is determined by a bijection from $V(F_0)$ to $V(Q)$,
there are at most $f!$ possibilities.
Hence, $|\mathcal P|\leq f!\,N_F(G)$,
and thus
$N_{F_0}(G)\leq f!\,N_F(G)$.
Consequently,
$N_F(G)\geq \delta_0 m^{f/2}/f!$.
Taking
$\delta=\delta_0/f!$
completes the proof of Case~2.
\end{proof}

\section{Proofs of Theorems \ref{thm1.6C}, \ref{thm1.5C}, and \ref{thm1.4C}}

For an edge subset $E'\subseteq E(G)$, we say that $H$ is the subgraph
of $G$ induced by $E'$ if $E(H)=E'$ and $V(H)$ is precisely the set of
endpoints of the edges in $E'$.
Thus, $H$ contains no isolated vertices.
In this section, we prove Theorem~\ref{thm1.6C},
which establishes edge-spectral supersaturation stability,
and then derive Theorems~\ref{thm1.5C} and~\ref{thm1.4C}
as applications.

\begin{proof}[\textbf{Proof of Theorem \ref{thm1.6C}}]
Since the assertion for a smaller positive
$\varepsilon$ implies that for a larger one,
it suffices to consider
$0<\varepsilon<10^{-2}$.
Fix such an $\varepsilon$.

\vspace{1mm}
(i) We first address the case when $r\geq3$ and $\rho(G)\geq\sqrt{(1\!-\!\frac{1}{r}\!-\!\delta)2m}$.
For convenience, we establish the following hierarchy:
\begin{align}\label{align-06FF}
0<\delta\ll \varepsilon_{\ref{def1.1}}
\ll \delta_{\ref{lem2.5C}}\lll \varepsilon_{\ref{lem2.5C}}
\ll \delta_{\ref{lem2.4C}}\lll \varepsilon_{\ref{lem2.4C}}
\ll \varepsilon\ll1,
\end{align}
where $\varepsilon_{\ref{def1.1}}$ is the parameter used in
Definition \ref{def1.1},
$\delta_{\ref{lem2.5C}}$ and $\varepsilon_{\ref{lem2.5C}}$
are taken from Lemma \ref{lem2.5C}, and
$\delta_{\ref{lem2.4C}}$ and $\varepsilon_{\ref{lem2.4C}}$
are the constants obtained from Lemma \ref{lem2.4C} applied with
$F=K_{r+1}$.

Applying Theorem \ref{thm1.3C} with its parameter $\varepsilon$
taken to be $\frac{1}{12}\varepsilon_{\ref{def1.1}}$, we obtain constants
$\delta_{\ref{thm1.3C}}>0$ and $m_{\ref{thm1.3C}}$ such that every graph
$H$ with $e(H)\geq m_{\ref{thm1.3C}}$ satisfying
$\rho(H)\geq \sqrt{\left(1-\frac{1}{r}+\frac{1}{12}\varepsilon_{\ref{def1.1}}\right)2e(H)}$
has $N_F(H)\geq \delta_{\ref{thm1.3C}} e(H)^{f/2}$.
By Lemma \ref{lem2.1C}, there exist
constants $\eta_{\ref{lem2.1C}}>0$ and $n_{\ref{lem2.1C}}$ such that every graph $H$ on at least
$n_{\ref{lem2.1C}}$ vertices satisfying
$N_F(H)\leq \eta_{\ref{lem2.1C}} |H|^f$
can be made $F$-free by deleting at most
$10^{-40}\varepsilon_{\ref{def1.1}}|H|^2$ edges.
Choose $\eta>0$ sufficiently small so that
\begin{align}\label{align-eta-i}
\eta(1-\varepsilon_{\ref{def1.1}})^{-f/2}
<\frac12\min\{\delta_{\ref{thm1.3C}},\eta_{\ref{lem2.1C}}\}.
\end{align}
Choose $\delta>0$ sufficiently small according to
\eqref{align-06FF}, and then choose $m_0$ sufficiently large so that
all size conditions required below are satisfied.

Let $G$ be a graph with $m=e(G)\geq m_0$ satisfying
$N_F(G)\leq\eta m^{f/2}$ and
$\rho(G)\geq\sqrt{(1-1/r-\delta)2m}$.
Construct $G_i$ for $1\leq i\leq\ell$ as described in
Definition \ref{def1.1}. From Definition \ref{def1.1}, we know that
$\ell\leq\lfloor\varepsilon_{\ref{def1.1}}m\rfloor$, and that
$G_1$ is obtained from $G$ by possibly deleting isolated vertices.
Since $e(G_\ell)=m-\ell+1\geq(1-\varepsilon_{\ref{def1.1}})m$,
we deduce from \eqref{align-eta-i} that
\begin{align*}
N_F(G_\ell)
\leq N_F(G)
\leq\eta m^{f/2}
\leq
\eta(1-\varepsilon_{\ref{def1.1}})^{-f/2}
e(G_\ell)^{f/2}
<\delta_{\ref{thm1.3C}} e(G_\ell)^{f/2}.
\end{align*}
Therefore, by the contrapositive of Theorem \ref{thm1.3C},
\begin{align}\label{align-07FF}
\rho(G_\ell)
<
\sqrt{\Big(
1-\frac1r+\frac{\varepsilon_{\ref{def1.1}}}{12}
\Big)2e(G_\ell)}.
\end{align}
Moreover, since
$e(G_\ell)\leq\frac12|G_\ell|^2$, it follows from
\eqref{align-eta-i} that
\begin{align}\label{align-07HF}
N_F(G_\ell)
\leq
\eta(1-\varepsilon_{\ref{def1.1}})^{-f/2}
e(G_\ell)^{f/2}
\leq
\eta(1-\varepsilon_{\ref{def1.1}})^{-f/2}
2^{-f/2}|G_\ell|^f
<\eta_{\ref{lem2.1C}}|G_\ell|^f.
\end{align}

Recall that
$\ell\leq\lfloor\varepsilon_{\ref{def1.1}}m\rfloor$.
We now claim that
$\ell<\lfloor\varepsilon_{\ref{def1.1}}m\rfloor$.
Assume, for the sake of contradiction, that
$\ell=\lfloor\varepsilon_{\ref{def1.1}}m\rfloor$.
Since $G_1$ is obtained from $G$ only by deleting isolated
vertices, $e(G_1)=m$ and $\rho(G_1)=\rho(G)$. Hence,
by choosing $\delta\ll\varepsilon_{\ref{def1.1}}$, we have
\begin{align}\label{align-08F}
\Phi(G_1)
=\frac{\rho(G_1)}{\sqrt{e(G_1)}}
\geq
\sqrt{2\Big(1-\frac1r-\delta\Big)}
\geq
\sqrt{2\Big(
1-\frac1r+\frac{\varepsilon_{\ref{def1.1}}}{6}
\Big)}
-\frac{\varepsilon_{\ref{def1.1}}}{6}.
\end{align}

Since $r\geq3$ and
$0<\varepsilon_{\ref{def1.1}}\ll1$, it follows from
\eqref{align-08F} that $\rho(G_1)\geq\sqrt m$.
Combining part (ii) of Lemma \ref{lem2.8C} with $\ell = \lfloor \varepsilon_{\ref{def1.1}}m \rfloor$, for sufficiently large $m$ we obtain
\begin{align*}
\Phi(G_\ell)
\geq \Phi(G_1)+\frac{\ell-1}{5m}
\geq \Phi(G_1)+\frac{\varepsilon_{\ref{def1.1}}}{6}.
\end{align*}
This, together with \eqref{align-08F}, yields
$\Phi(G_\ell)
\geq\sqrt{2\big(1-1/r+\frac{\varepsilon_{\ref{def1.1}}}{6}\big)}$.
Thus
$\rho(G_\ell)\geq
\sqrt{\big(1-1/r+\frac{\varepsilon_{\ref{def1.1}}}{6}\big)2e(G_\ell)}$,
which contradicts \eqref{align-07FF}. Therefore,
$\ell<\lfloor\varepsilon_{\ref{def1.1}}m\rfloor$,
as claimed.

Since
$|G_\ell| \geq \sqrt{2e(G_\ell)}
\geq \sqrt{2(1-\varepsilon_{\ref{def1.1}})m}$,
we have $|G_\ell|\geq n_{\ref{lem2.1C}}$ for sufficiently large $m_0$.
Thus, by \eqref{align-07HF} and Lemma \ref{lem2.1C}, we can remove at most
$10^{-40}\varepsilon_{\ref{def1.1}}|G_\ell|^2$ edges from
$G_\ell$ to obtain an $F$-free graph $G'$.
Since $|G'|=|G_\ell|\geq \sqrt{2e(G_\ell)}=\Omega(\sqrt m)$,
Lemma \ref{lem2.3C} applies to $G'$ for sufficiently large $m_0$.
Thus, we can remove at most
$10^{-40}\varepsilon_{\ref{def1.1}}|G_\ell|^2$
additional edges from $G'$ to obtain a $K_{r+1}$-free graph
$G''$.
Let $G'''$ be the subgraph induced by
$E(G)\setminus E(G'')$. Then
\begin{align}\label{align-09F}
e(G''')
\leq
(\ell-1)
+2\times10^{-40}\varepsilon_{\ref{def1.1}}|G_\ell|^2
<
\varepsilon_{\ref{def1.1}}m
+2\times10^{-40}
\varepsilon_{\ref{def1.1}}|G_\ell|^2.
\end{align}

By Definition \ref{def1.1} and the fact that
$\ell<\lfloor\varepsilon_{\ref{def1.1}}m\rfloor$,
the graph $G_\ell$ contains no light edges with respect to the
terminal eigenvector $\bm{x}^{(\ell)}$.
Moreover, since $G_\ell$ is induced by its edge set,
it has no isolated vertices.
Define $a:=1.1$. Since $r\geq3$ and
$\delta\ll1$, we have
$\rho(G)
\geq\sqrt{(1-\frac1r-\delta)2m}
\geq\sqrt{am}$.
By part {\rm (i)} of Lemma \ref{lem2.8C},
$\rho(G_\ell)\geq\sqrt{a\,e(G_\ell)}$.
Therefore, applying part {\rm (iii)} of Lemma \ref{lem2.8C}
to $G_\ell$, with $\bm{x}^{(\ell)}$ as the corresponding
non-negative unit eigenvector, we obtain
$|G_\ell|
\leq \big(\frac{8}{a-1}\big)^{10}\sqrt{e(G_\ell)}
\leq 10^{20}\sqrt m$.
Combining this with \eqref{align-09F} yields
$d(G,G'')=e(G''') \leq 3\varepsilon_{\ref{def1.1}}m$.
Hence,
$\rho(G''') \leq\sqrt{2e(G''')} \leq\sqrt{6\varepsilon_{\ref{def1.1}}m}$.
Consequently,
\begin{align*}
\rho(G'')
\geq
\rho(G)-\rho(G''')
\geq
\sqrt{\Bigl(1-\frac1r-\delta\Bigr)2m}
-\sqrt{6\varepsilon_{\ref{def1.1}}m}
\geq
\sqrt{\Bigl(1-\frac1r-6\sqrt{\varepsilon_{\ref{def1.1}}}\Bigr)2m},
\end{align*}
where the last inequality follows from
$\delta\ll\varepsilon_{\ref{def1.1}}\ll1$.
In view of $e(G'')\leq m$ and \eqref{align-06FF}, we further obtain
\[
\rho(G'')
\geq
\sqrt{\Bigl(1-\frac1r-6\sqrt{\varepsilon_{\ref{def1.1}}}\Bigr)2e(G'')}
\geq\sqrt{(1-\frac1r-\delta_{\ref{lem2.5C}})2e(G'')}.
\]
Moreover,
$e(G'')\geq m-d(G,G'') \geq (1-3\varepsilon_{\ref{def1.1}})m$,
so $e(G'')$ is sufficiently large when $m_0$ is sufficiently large.
Therefore, Lemma \ref{lem2.5C} applies and yields an induced
subgraph $G^{(4)}\subseteq G''$ such that
\begin{align}\label{align-010F}
|G^{(4)}|
\leq
(1+\varepsilon_{\ref{lem2.5C}})
\sqrt{\frac{2e(G'')}{1-1/r}}
\leq \sqrt{\frac{(1+3\varepsilon_{\ref{lem2.5C}})2m}{1-1/r}},
\end{align}
and
$d(G'',G^{(4)})
\leq \varepsilon_{\ref{lem2.5C}}e(G'')
\leq \varepsilon_{\ref{lem2.5C}}m$.
Let $G^{(5)}$ be the subgraph induced by
$E(G)\setminus E(G^{(4)})$. It is easy to see that
\begin{align}\label{align-011F}
e(G^{(5)})
=d(G,G^{(4)})
\leq
d(G,G'')+d(G'',G^{(4)})
\leq
\left(3\varepsilon_{\ref{def1.1}}+\varepsilon_{\ref{lem2.5C}}\right)m.
\end{align}
Recall that
$\varepsilon_{\ref{def1.1}}
\ll\delta_{\ref{lem2.5C}}
\lll\varepsilon_{\ref{lem2.5C}}$.
Then, from \eqref{align-011F}, we have
$e(G^{(5)}) \leq2\varepsilon_{\ref{lem2.5C}}m$,
and hence
$\rho(G^{(5)})\leq\sqrt{2e(G^{(5)})}
\leq\sqrt{4\varepsilon_{\ref{lem2.5C}}m}$.
Combining this with
$\rho(G)\geq
\sqrt{(1-\frac1r-\delta)2m}$, we obtain
\begin{align*}
\rho(G^{(4)})
&\geq
\rho(G)-\rho(G^{(5)})
\geq
\sqrt{\Big(1-\frac1r-\delta\Big)2m}
-\sqrt{4\varepsilon_{\ref{lem2.5C}}m}\\
&\geq
\sqrt{\Big(
1-\frac1r-6\sqrt{\varepsilon_{\ref{lem2.5C}}}
\Big)2m}
\geq
\Big(1-\frac1r -9\varepsilon_{\ref{lem2.5C}}^{1/4}\Big)|G^{(4)}|,
\end{align*}
where the final inequality follows from
\eqref{align-010F} and a straightforward calculation.

Recall that
$\varepsilon_{\ref{lem2.5C}}
\ll\delta_{\ref{lem2.4C}}$.
Hence, we may assume that
$9\varepsilon_{\ref{lem2.5C}}^{1/4}<
\delta_{\ref{lem2.4C}}$.
Write $n:=|G^{(4)}|$. Then
$\rho(G^{(4)})>\left(1-\frac1r-\delta_{\ref{lem2.4C}}\right)n$.
Moreover, the preceding lower bound on $\rho(G^{(4)})$ gives
$n=|G^{(4)}|\geq\rho(G^{(4)})=\Omega(\sqrt m)$.
Thus, $n$ is sufficiently large when $m_0$ is sufficiently large.

Since $G^{(4)}\subseteq G''$ and $G''$ is $K_{r+1}$-free,
the graph $G^{(4)}$ is also $K_{r+1}$-free.
Applying Lemma \ref{lem2.4C} with $F=K_{r+1}$, we obtain a
Tur\'an graph $T_{n,r}$ on the same vertex set as $G^{(4)}$
such that
$d(G^{(4)},T_{n,r})\leq\varepsilon_{\ref{lem2.4C}}n^2$.
Recall that $G^{(4)}\subseteq G_\ell$ and
$|G_\ell|\leq10^{20}\sqrt m$. It follows that
$n\leq10^{20}\sqrt m$. Combining this with
\eqref{align-011F} gives
\begin{align*}
d(G,T_{n,r})
\leq
d(G,G^{(4)})+d(G^{(4)},T_{n,r})
\leq
\left(
3\varepsilon_{\ref{def1.1}}
+\varepsilon_{\ref{lem2.5C}}
\right)m
+\varepsilon_{\ref{lem2.4C}}n^2
\leq
\varepsilon m,
\end{align*}
where the last inequality follows from \eqref{align-06FF}.
This completes the proof of part {\rm (i)}.

\vspace{1mm}
{\rm (ii)}
Suppose first that $r=2$ and
$\rho(G)\geq\sqrt{(1-\delta)m}$.
It suffices to prove that
$d(G,K_{U,V})\leq\varepsilon m$ for some disjoint
$U,V\subseteq V(G)$.
For simplicity, we establish the following hierarchy:
\begin{align}\label{align-012F}
0<\delta
\ll \delta_{\ref{lem2.7C}}\lll \varepsilon_{\ref{lem2.7C}}
\ll \delta_{\ref{lem2.6C}}\lll \varepsilon_{\ref{lem2.6C}}
\ll \varepsilon_{\ref{def1.1}}
\ll \varepsilon\ll1,
\end{align}
where
$\varepsilon_{\ref{def1.1}}$ is the parameter used in
Definition \ref{def1.1},
$\delta_{\ref{lem2.6C}}$ and $\varepsilon_{\ref{lem2.6C}}$
are obtained from Lemma \ref{lem2.6C}, and
$\delta_{\ref{lem2.7C}}$ and $\varepsilon_{\ref{lem2.7C}}$
are taken from Lemma \ref{lem2.7C}.

Applying Theorem \ref{thm1.3C} with its parameter $\varepsilon$
taken to be $\frac{1}{12}\varepsilon_{\ref{def1.1}}$,
there exist constants $\delta_{\ref{thm1.3C}}>0$ and $m_{\ref{thm1.3C}}$ such that every graph
$H$ with $q=e(H)\geq m_{\ref{thm1.3C}}$ satisfying
$\rho(H)
\geq \sqrt{\left(1+\frac{\varepsilon_{\ref{def1.1}}}{6}\right)q}$
has
$N_F(H)\geq \delta_{\ref{thm1.3C}} q^{f/2}$.
Applying Lemma \ref{lem2.1C} with its parameter $\varepsilon$
taken to be $\delta_{\ref{lem2.7C}}^7$,
there exist constants $\eta_{\ref{lem2.1C}}>0$ and $n_{\ref{lem2.1C}}$ such that
every graph $H$ on at least $n_{\ref{lem2.1C}}$ vertices satisfying
$N_F(H)\leq \eta_{\ref{lem2.1C}}|H|^f$
can be made $F$-free by deleting at most
$\delta_{\ref{lem2.7C}}^7 |H|^2$ edges.
Choose $\eta>0$ sufficiently small so that
\begin{align}\label{align-eta-ii}
\eta(1-\varepsilon_{\ref{def1.1}})^{-f/2}
<
\frac12\min\{\delta_{\ref{thm1.3C}},\eta_{\ref{lem2.1C}}\}.
\end{align}
Choose $\delta>0$ sufficiently small according to
\eqref{align-012F}, and then choose $m_0$ sufficiently large so that
all size conditions and estimates below hold.

Let $G$ be a graph with $m=e(G)\geq m_0$ satisfying
$N_F(G)\leq\eta m^{f/2}$ and
$\rho(G)\geq\sqrt{(1-\delta)m}$.
Construct $G_i$ for $1\leq i\leq\ell$ as in
Definition \ref{def1.1}. Recall that
$\ell\leq\lfloor\varepsilon_{\ref{def1.1}}m\rfloor$ and that
$G_1$ is obtained from $G$ by possibly deleting isolated vertices.
Since $e(G_\ell)\geq(1-\varepsilon_{\ref{def1.1}})m$,
we have
\begin{align*}
N_F(G_\ell)
\leq \eta m^{f/2}
\leq \eta(1-\varepsilon_{\ref{def1.1}})^{-f/2}e(G_\ell)^{f/2}
<\delta_{\ref{thm1.3C}} e(G_\ell)^{f/2}.
\end{align*}
Hence, by the contrapositive of Theorem \ref{thm1.3C},
\begin{align}\label{align-07F-bip}
\rho(G_\ell)
<
\sqrt{\left(
1+\frac{\varepsilon_{\ref{def1.1}}}{6}
\right)e(G_\ell)}.
\end{align}
Moreover, since
$e(G_\ell)\leq\frac12|G_\ell|^2$, we obtain
\begin{align}\label{align-07H-bip}
N_F(G_\ell)<\eta_{\ref{lem2.1C}} |G_\ell|^f.
\end{align}

We now claim that
$\ell<\lfloor\varepsilon_{\ref{def1.1}}m\rfloor$.
Assume, for the sake of contradiction, that
$\ell=\lfloor\varepsilon_{\ref{def1.1}}m\rfloor$.
Note that $\rho(G_1)=\rho(G)\geq\sqrt{(1-\delta)m}$.
Since $\delta\ll\varepsilon_{\ref{def1.1}}$, we have
$\rho(G_1)\geq\left(1-\frac{\varepsilon_{\ref{def1.1}}}{60}\right)\sqrt m$.
By part {\rm (ii)} of Lemma \ref{lem2.8C} and the assumption
$\ell=\lfloor\varepsilon_{\ref{def1.1}}m\rfloor$, for sufficiently
large $m$ we derive
\begin{align*}
\Phi(G_\ell)
\geq
\Phi(G_1)+\frac{\ell-1}{5m}
\geq
1-\frac{\varepsilon_{\ref{def1.1}}}{60}
+\frac{11\varepsilon_{\ref{def1.1}}}{60}
=1+\frac{\varepsilon_{\ref{def1.1}}}{6}.
\end{align*}
It follows that
\[
\rho(G_\ell)
\geq
\left(
1+\frac{\varepsilon_{\ref{def1.1}}}{6}
\right)\sqrt{e(G_\ell)}
>
\sqrt{\left(
1+\frac{\varepsilon_{\ref{def1.1}}}{6}
\right)e(G_\ell)},
\]
contradicting \eqref{align-07F-bip}.
Therefore,
$\ell<\lfloor\varepsilon_{\ref{def1.1}}m\rfloor$,
as claimed.

Now let $H=G_\ell$ and $h=e(H)$, and let
$\bm{x}:=\bm{x}^{(\ell)}=(x_v)_{v\in V(H)}$
be the non-negative unit eigenvector corresponding to $\rho(H)$
chosen at the terminal step of the procedure in Definition \ref{def1.1}.
Define $a:=1-\delta$. Since $\rho(G)\geq\sqrt{(1-\delta)m}=\sqrt{am}$,
by part {\rm (i)} of Lemma \ref{lem2.8C}, we have
$\rho(H)\geq\sqrt{ah}=\sqrt{(1-\delta)h}$.
We divide the remainder of the proof into two cases.

\medskip
\noindent
\textbf{Case 1.
$\max\{x_u:u\in V(H)\}>\delta_{\ref{lem2.7C}}^{-1}h^{-1/4}$.}

Since $\delta\ll\delta_{\ref{lem2.7C}}$, we have
$\rho(H) \geq \sqrt{(1-\delta_{\ref{lem2.7C}})h}$.
By Lemma \ref{lem2.7C}, there exist disjoint vertex subsets
$U,V\subseteq V(H)$ such that
$d(H,K_{U,V})\leq\varepsilon_{\ref{lem2.7C}}h$.
On the other hand,
$d(G,H)=\ell-1<\varepsilon_{\ref{def1.1}}m$.
Combining these with \eqref{align-012F}, we obtain
\begin{align*}
d(G,K_{U,V})
\leq
d(G,H)+d(H,K_{U,V})
\leq \varepsilon_{\ref{def1.1}}m+\varepsilon_{\ref{lem2.7C}}h
\leq \varepsilon m.
\end{align*}

\medskip
\noindent
\textbf{Case 2.}
$\max\{x_u:u\in V(H)\}\leq \delta_{\ref{lem2.7C}}^{-1}h^{-1/4}$.

By Definition \ref{def1.1} and the fact that
$\ell<\lfloor\varepsilon_{\ref{def1.1}}m\rfloor$,
the procedure terminates because $H$ contains no light edges
with respect to $\bm{x}=\bm{x}^{(\ell)}$. Hence
$x_ux_v>\frac18h^{-1/2}$ for every $uv\in E(H)$.
Moreover, $H$ has no isolated vertices. Hence, together with
$\max\{x_u:u\in V(H)\}
\leq \delta_{\ref{lem2.7C}}^{-1}h^{-1/4}$,
we obtain
$\min_{v\in V(H)}x_v
>\frac18\delta_{\ref{lem2.7C}}h^{-1/4}$.
Combining this with
$\sum_{u\in V(H)}x_u^2=1$, we deduce that
$|H|\leq 64\delta_{\ref{lem2.7C}}^{-2}\sqrt h$.

Since $|H|\geq\sqrt{2h}\geq\sqrt{2(1-\varepsilon_{\ref{def1.1}})m}$,
we have $|H|\geq n_{\ref{lem2.1C}}$ for sufficiently large $m_0$.
By \eqref{align-07H-bip} and the preceding argument, we can remove at most
$\delta_{\ref{lem2.7C}}^7|H|^2$ edges from $H$ to obtain an
$F$-free graph $H'$.
Since $|H'|=|H|\geq \sqrt{2h}=\Omega(\sqrt m)$,
we have $|H'|$ sufficiently large whenever $m_0$ is sufficiently large.
Hence, applying Lemma \ref{lem2.3C} to $H'$ with
$\varepsilon=\delta_{\ref{lem2.7C}}^7$,
we may remove at most
$\delta_{\ref{lem2.7C}}^7|H|^2$
additional edges from $H'$ to obtain a triangle-free graph $H''$.
Let $H'''$ denote the subgraph induced by
$E(H)\setminus E(H'')$. Then
$$d(H,H'')=e(H''')\leq2\delta_{\ref{lem2.7C}}^7|H|^2
\leq \frac18\delta_{\ref{lem2.7C}}^2h,$$
where the last inequality follows by choosing
$\delta_{\ref{lem2.7C}}$ sufficiently small.
It follows that
$\rho(H''')
\leq \sqrt{2e(H''')}
\leq \frac12\delta_{\ref{lem2.7C}}\sqrt h$.
Therefore,
\begin{align*}
\rho(H'')
\geq
\rho(H)-\rho(H''')
\geq
\sqrt{(1-\delta)h}
-\frac12\delta_{\ref{lem2.7C}}\sqrt h
\geq
\sqrt{\left(
1-2\delta_{\ref{lem2.7C}}
\right)h}
>
\sqrt{\left(
1-\delta_{\ref{lem2.6C}}
\right)e(H'')},
\end{align*}
where the last inequality follows from
\eqref{align-012F} and $h\geq e(H'')$.

By Lemma \ref{lem2.6C}, there exist two disjoint subsets
$U,V\subseteq V(H'')$ such that
$d(H'',K_{U,V})
\leq
\varepsilon_{\ref{lem2.6C}}e(H'')
\leq
\varepsilon_{\ref{lem2.6C}}m$.
Therefore,
\begin{align*}
d(G,K_{U,V})
&\leq
d(G,H)+d(H,H'')+d(H'',K_{U,V})\\
&\leq
\varepsilon_{\ref{def1.1}}m
+\frac18\delta_{\ref{lem2.7C}}^2h
+\varepsilon_{\ref{lem2.6C}}m
\leq \varepsilon m,
\end{align*}
as required. This proves part {\rm (ii)} when $r=2$.

\medskip
It remains to consider the case
$r=1$ and $\beta'(F)\geq2$.
Let $A$ be an independent vertex cover of $F$ of minimum size.
Set $a:=|A|=\beta'(F)$ and $b:=f-a$.
Since $A$ is a vertex cover, $V(F)\setminus A$ is independent.
Thus $F$ is a spanning subgraph of $K_{a,b}$.
Let $F_0$ be the graph obtained from $K_{a,b}$ by adding an edge
inside the part of size $a$.
Since $\chi(F)=2$ and $A$ is an independent vertex cover, we have
$b\geq1$. Thus, as $a\geq2$, the added edge together with any
vertex in the other part forms a triangle in $F_0$. Since $F_0$
is clearly $3$-colorable, we conclude that $\chi(F_0)=3$.
Clearly, $F$ is a spanning subgraph of $F_0$.
Set $C_F:=N_{F_0}(K_f)$.
For every copy of $F$ in $G$, there are at most $C_F$ copies of
$F_0$ containing it. Hence
$N_{F_0}(G)\leq C_FN_F(G)$.

Apply the already proved $r=2$ case to $F_0$ with parameter
$\varepsilon$. Then there exist constants
$\eta_0=\eta_0(F_0,\varepsilon)>0$,
$\delta_0=\delta_0(F_0,\varepsilon)>0$, and $m_0'$ such that
every graph $J$ with $q=e(J)\geq m_0'$ satisfying
$N_{F_0}(J)\leq\eta_0q^{f/2}$
and $\rho(J)\geq\sqrt{(1-\delta_0)q}$
admits two disjoint subsets $U,V\subseteq V(J)$ satisfying
$d(J,K_{U,V})\leq \varepsilon q$.

Choose $\eta\leq\frac{\eta_0}{C_F}$,
$\delta\leq\delta_0$,
and $m_0\geq m_0'$.
Let $G$ satisfy the assumptions of Theorem \ref{thm1.6C}.
Since $N_{F_0}(G)\leq C_FN_F(G)$, we have
\[
N_{F_0}(G)
\leq
C_FN_F(G)
\leq
C_F\eta m^{f/2}
\leq
\eta_0m^{f/2}.
\]
Moreover,
$\rho(G)
\geq \sqrt{(1-\delta)m}
\geq \sqrt{(1-\delta_0)m}$.
Applying the $r=2$ case to $F_0$, we obtain two disjoint subsets
$U,V\subseteq V(G)$ such that
$d(G,K_{U,V})\leq \varepsilon m$,
as required.

This completes the proof of Theorem \ref{thm1.6C}.
\end{proof}

\begin{proof}[\textbf{Proof of Theorem \ref{thm1.5C}}]
Since the assertion for a smaller positive $\varepsilon$ implies that for a larger one,
it suffices to consider $0<\varepsilon\ll 1$.
Choose $\varepsilon_{\ref{thm1.6C}}>0$ sufficiently small such that
$5\sqrt{\varepsilon_{\ref{thm1.6C}}}<\varepsilon$.
Let
$\eta_{\ref{thm1.6C}},\delta_{\ref{thm1.6C}}>0$
and $m_{\ref{thm1.6C}}$
be the constants obtained from Theorem \ref{thm1.6C}
with parameter $\varepsilon_{\ref{thm1.6C}}$.
For convenience, we establish the following hierarchy:
\begin{align}\label{align-013F}
0<\delta\ll \delta_{\ref{thm1.6C}}
\lll \varepsilon_{\ref{thm1.6C}}\ll \varepsilon\ll 1.
\end{align}

Apply Lemma \ref{lem2.1C} with parameter $\delta/2$,
and let $\eta_{\ref{lem2.1C}}>0$ be the corresponding constant.
Choose $\eta>0$ sufficiently small such that
\begin{align}\label{align-013G}
\eta\leq \eta_{\ref{lem2.1C}}
\quad\text{and}\quad
\eta\leq
\eta_{\ref{thm1.6C}}
\Big(\frac{1-1/r-\delta}{2}\Big)^{f/2}.
\end{align}

Let $n$ be sufficiently large and let $G$ be a graph with $n$ vertices
and $m$ edges such that
$N_F(G)\leq \eta n^f$ and
$m\geq\big(1-1/r-\delta\big)\frac{n^2}{2}$.
We further assume that $n$ is sufficiently large so that Lemma \ref{lem2.1C} and the Erd\H{o}s--Stone theorem used below are applicable, and that
$(1-1/r-\delta)\frac{n^2}{2}\ge m_{\ref{thm1.6C}}$.
We first verify that all the hypotheses of Theorem \ref{thm1.6C}
are satisfied.


By Lemma \ref{lem2.1C},
removing at most $\frac12\delta n^2$ edges from \(G\) results in an \(F\)-free graph
$G'$ with $m'$ edges. Since $G'$ is \(F\)-free,
the Erd\H{o}s--Stone theorem implies $m'\leq (1-\frac{1}{r}+\delta)\frac{n^2}{2}$.
Thus,
\begin{align}\label{align-014F}
(1-\frac{1}{r}-\delta)\frac{n^2}{2}\leq m\leq m'+\frac12\delta n^2\leq (1-\frac{1}{r}+2\delta)\frac{n^2}{2}\leq \frac12n^2.
\end{align}
Moreover, by \eqref{align-013G} and the left-hand inequality in
\eqref{align-014F}, we have
\begin{align}\label{align-014H}
N_F(G)
\leq \eta n^f
\leq
\eta_{\ref{thm1.6C}}
\left(\frac{1-1/r-\delta}{2}\right)^{f/2}n^f
\leq \eta_{\ref{thm1.6C}}m^{f/2}.
\end{align}
In light of \eqref{align-014F}, we also derive that
 $(1-\frac{1}{r}+2\delta)\frac{n^2}{2}\geq m$
and $\rho(G)\geq \rho(G')\geq \frac{2m'}{n}\geq(1-\frac{1}{r}-2\delta)n$.
Therefore, we deduce that
\begin{align}\label{align-014G}
\rho(G)\geq  (1\!-\!\frac{1}{r}\!-\!2\delta)n
\geq \big(1\!-\!\frac{1}{r}\!-\!2\delta\big)\sqrt{\frac{2m}{1\!-\!{1}/{r}\!+\!2\delta}}
\geq \sqrt{\big(1\!-\!\frac{1}{r}\!-\!\frac12\delta_{\ref{thm1.6C}}\big)2m},
\end{align}
where the last inequality follows from $\delta\ll\delta_{\ref{thm1.6C}},$ as assumed in \eqref{align-013F}.
Together with \eqref{align-014H} and the choice of $n$,
this verifies all the hypotheses needed to apply
Theorem \ref{thm1.6C}.
We now distinguish the cases $r\geq3$ and $r=2$.

We first consider the case $r\geq 3$.
Given \eqref{align-014H}, \eqref{align-014G}, and $m\geq m_{\ref{thm1.6C}}$,
it follows from part (i) of Theorem \ref{thm1.6C} that there exists a Tur\'{a}n graph $T_{n',r}$
satisfying $V(T_{n',r})\subseteq V(G)$ and $d(G,T_{n',r})\leq \varepsilon_{\ref{thm1.6C}} m$.
Hence,
$e(T_{n',r})\geq (1-\varepsilon_{\ref{thm1.6C}})m.$
It is also known that $e(T_{n',r})\leq \frac{r-1}{2r}{n'}^2$.
Combining this with the left-hand inequality in \eqref{align-014F}, we deduce that
\begin{align*}
\frac{r-1}{2r}n'^2\geq (1-\varepsilon_{\ref{thm1.6C}})m\geq (1-\varepsilon_{\ref{thm1.6C}})(1-\frac{1}{r}-\delta)\frac{n^2}{2}.
\end{align*}
Together with the condition $\delta\ll\varepsilon_{\ref{thm1.6C}}$,
this inequality implies that $n'\geq (1-\varepsilon_{\ref{thm1.6C}})n$.

Recall that $|G|=n$ and $V(T_{n',r})\subseteq V(G)$.
We may construct a graph $T^*$ isomorphic to $T_{n,r}$ such that
$V(T^*)=V(G)$ and $T_{n',r}\subseteq T^*$.
Clearly, $d(T^*,T_{n',r})\leq (n-n')n\leq\varepsilon_{\ref{thm1.6C}} n^2$ since $n'\geq (1-\varepsilon_{\ref{thm1.6C}})n$.
Moreover, we recall that $d(G,T_{n',r})\leq \varepsilon_{\ref{thm1.6C}}m$, and $m\leq \frac12 n^2$ by \eqref{align-014F}.
Thus,
   $$d(G,T^*)\leq d(G,T_{n',r})+d(T_{n',r},T^*)
   \leq \frac32 \varepsilon_{\ref{thm1.6C}} n^2
   <\varepsilon n^2.$$
Therefore, $T^*$ is the desired Tur\'{a}n graph for $r\geq3$.

Next, we consider the case $r=2$.
Given \eqref{align-014H}, \eqref{align-014G}, and
$m\geq m_{\ref{thm1.6C}}$,
it follows from part (ii) of Theorem \ref{thm1.6C} that
there exist disjoint subsets $U,V\subseteq V(G)$
such that $d(G,K_{U,V})\leq \varepsilon_{\ref{thm1.6C}} m$.
Thus,
$|U|\cdot |V|\geq(1-\varepsilon_{\ref{thm1.6C}})m.$
Combining this with the left-hand inequality in \eqref{align-014F}, we deduce that
\begin{align}\label{align-015F}
|U|\cdot |V|\geq\big(1-\varepsilon_{\ref{thm1.6C}}\big)\big(\frac{1}{2}-\delta\big)\frac{n^2}{2}
\geq\big(\frac14-\varepsilon_{\ref{thm1.6C}}\big)n^2.
\end{align}
Since $U$ and $V$ are disjoint subsets of $V(G)$,
we have $|U|+|V|\leq n$.
Combining this with \eqref{align-015F}, we derive
$(\frac12\!-\!\sqrt{\varepsilon_{\ref{thm1.6C}}})n\leq |U|,|V|\leq (\frac12\!+\!\sqrt{\varepsilon_{\ref{thm1.6C}}})n$.
Assume without loss of generality that $|U|\leq |V|$,
and choose a partition $U'\cup V'$ of $V(G)$ such that
$|U'|=\lfloor\frac{n}{2}\rfloor$, $U\subseteq U'$, and

\vspace{2mm}
$\bullet$ if $|V|\geq\lceil \frac{n}{2}\rceil$, then $V'\subseteq V$;

\vspace{2mm}
$\bullet$ if $|V|<\lceil \frac{n}{2}\rceil$, then $V\subseteq V'$.

\vspace{2mm}
\noindent
Clearly, $\big||V'|-|V|\big|\leq 2\sqrt{\varepsilon_{\ref{thm1.6C}}} n$ and $\big||U'|-|U|\big|\leq 2\sqrt{\varepsilon_{\ref{thm1.6C}}} n$.
Since $U\subseteq U'$ and one of $V,V'$ is contained in the other,
we have
$|U'\triangle U|+|V'\triangle V|\leq4\sqrt{\varepsilon_{\ref{thm1.6C}}}\,n$.
Consequently,
\[
d(K_{U',V'},K_{U,V})
\leq
\big(|U'\triangle U|+|V'\triangle V|\big)n
\leq
4\sqrt{\varepsilon_{\ref{thm1.6C}}}\,n^2.
\]
Recalling that $d(G,K_{U,V})\leq \varepsilon_{\ref{thm1.6C}} m\leq \frac12 \varepsilon_{\ref{thm1.6C}}n^2\leq\sqrt{\varepsilon_{\ref{thm1.6C}}} n^2$,
we deduce that
$$d(G,K_{U',V'})\leq d(G,K_{U,V})+d(K_{U,V},K_{U',V'})\leq5\sqrt{\varepsilon_{\ref{thm1.6C}}}n^2<\varepsilon n^2.$$
Therefore, $K_{U',V'}$ is the desired Tur\'{a}n graph for $r=2$, thus completing the proof.
\end{proof}

\begin{proof}[\textbf{Proof of Theorem \ref{thm1.4C}}]
Since the assertion for a smaller positive value of $\varepsilon$ implies the corresponding assertion for any larger value, it suffices to consider $0<\varepsilon\ll1$.
Apply Theorem \ref{thm1.5C} with parameter $\varepsilon$, and let
$\eta_{\ref{thm1.5C}},\delta_{\ref{thm1.5C}}>0$ and
$n_{\ref{thm1.5C}}$ be the corresponding constants.
By decreasing $\delta_{\ref{thm1.5C}}$ if necessary, we may further assume that
$\delta_{\ref{thm1.5C}}<1-1/r$.
Next, apply Theorem \ref{thm1.3C} with its parameter $\varepsilon$
taken to be
$\frac14\delta_{\ref{thm1.5C}}^2$.
Let $\delta_{\ref{thm1.3C}}>0$ and $m_{\ref{thm1.3C}}$ be such that
every graph $J$ with $e(J)\ge m_{\ref{thm1.3C}}$ and
$\rho(J)\ge \sqrt{(1-\frac1r+\frac14\delta_{\ref{thm1.5C}}^2)2e(J)}$
satisfies
$N_F(J)\ge \delta_{\ref{thm1.3C}}e(J)^{f/2}$.
Choose $\delta>0$ sufficiently small such that
\begin{align}\label{align-016F}
(1+\sqrt2)\delta\le \frac12\delta_{\ref{thm1.5C}}.
\end{align}
Apply Lemma \ref{lem2.1C} with its parameter $\varepsilon$
taken to be $\delta^2$, and let
$\eta_{\ref{lem2.1C}}>0$ and $n_{\ref{lem2.1C}}$ be the corresponding constants.
Finally, choose $\eta>0$ sufficiently small such that
\begin{align}\label{align-016G}
\eta\le
\min\Big\{
\eta_{\ref{thm1.5C}},
\eta_{\ref{lem2.1C}},
\frac{\delta_{\ref{thm1.3C}}}{2}
\Big(
\frac{\big(1-1/r-\delta\big)^2}{2}
\Big)^{f/2}
\Big\}.
\end{align}

Let $n$ be sufficiently large, and let $G$ be a graph with $n$ vertices and $m$ edges such that
$N_F(G)\le \eta n^f$ and
$\rho(G)\geq (1-\frac{1}{r}-\delta)n$.
We further assume that $n$ is sufficiently large so that
$n\ge \max\{n_{\ref{thm1.5C}},n_{\ref{lem2.1C}}\}$ and
$\frac12(1-1/r-\delta)^2n^2\ge m_{\ref{thm1.3C}}$.
Our goal is to prove that there exists an $r$-partite Tur\'{a}n graph
$T_{n,r}$ such that $V(T_{n,r})=V(G)$ and
$d(G,T_{n,r})\leq \varepsilon n^2$.
In view of \eqref{align-016G}, we have
$N_F(G)\leq \eta_{\ref{thm1.5C}}n^f$.
Thus, by Theorem \ref{thm1.5C}, it suffices to show that
\begin{align}\label{align-017F}
m\geq
\Big(1-\frac{1}{r}-\delta_{\ref{thm1.5C}}\Big)
\frac{n^2}{2}.
\end{align}

By \eqref{align-016G}, we have
$N_F(G)\leq \eta_{\ref{lem2.1C}} n^f$.
Thus, by the preceding argument, one can remove at most
$\delta^2n^2$ edges from $G$ to obtain an $F$-free graph $G'$ with
$m'$ edges.
Let $H$ be the subgraph of $G$ induced by those removed edges.
Then $\rho(H)\leq\sqrt{2e(H)}\leq\sqrt{2}\,\delta n$.
Combining this with \eqref{align-016F} and the assumption
$\rho(G)\geq(1-\frac1r-\delta)n$, we obtain
\[
\rho(G')
\geq\rho(G)-\rho(H)
\geq
\Big(1-\frac1r-\delta-\sqrt2\,\delta\Big)n
\geq
\Big(1-\frac1r-\frac12\delta_{\ref{thm1.5C}}\Big)n.
\]

On the other hand, by the well-known inequality $\rho^2(G)\leq 2m$ and the assumption $\rho(G)\geq (1-\frac{1}{r}-\delta)n$, we have
$m\geq \frac12(1-{1}/{r}-\delta)^2n^2=\Omega(n^2).$
Since trivially $m\leq \binom{n}{2}$, it follows that $m=\Theta(n^2)$.
Moreover, by \eqref{align-016G},
\begin{align*}
N_F(G)
\leq \eta n^f
\leq
\frac{\delta_{\ref{thm1.3C}}}{2}
\Big(\frac{(1-1/r-\delta)^2}{2}\Big)^{f/2}n^f
\leq
\frac{\delta_{\ref{thm1.3C}}}{2}m^{f/2}
<
\delta_{\ref{thm1.3C}}m^{f/2}.
\end{align*}
Hence, by the contrapositive of Theorem \ref{thm1.3C},
$\rho(G)<\sqrt{(1-\frac1r+\frac14\delta_{\ref{thm1.5C}}^2)2m}$.
As $G'\subseteq G$, we also have $\rho(G')\leq\rho(G)$, and consequently
$\rho(G')\leq \sqrt{(1-\frac{1}{r}+\frac14\delta_{\ref{thm1.5C}}^2)2m}.$
Combining this with the preceding lower bound, we obtain
$(1-\frac1r-\frac12\delta_{\ref{thm1.5C}})^2 n^2
\le (1-\frac1r+\frac14\delta_{\ref{thm1.5C}}^2)2m$.
Consequently,
\[
m\ge
\frac{
\left(1-\frac1r-\frac12\delta_{\ref{thm1.5C}}\right)^2
}{
1-\frac1r+\frac14\delta_{\ref{thm1.5C}}^2
}
\frac{n^2}{2}
\ge
\left(1-\frac1r-\delta_{\ref{thm1.5C}}\right)\frac{n^2}{2}.
\]
Indeed, the last inequality follows from
\[
\Big(1-\frac1r-\frac12\delta_{\ref{thm1.5C}}\Big)^2
-\Big(1-\frac1r-\delta_{\ref{thm1.5C}}\Big)
\Big(1-\frac1r+\frac14\delta_{\ref{thm1.5C}}^2\Big)
=\frac{\delta_{\ref{thm1.5C}}^2}{4}\Big(\frac1r+\delta_{\ref{thm1.5C}}\Big)>0.
\]
Thus, \eqref{align-017F} follows, and the proof is complete.
\end{proof}

\section{Proof of Theorem \ref{thm1.8C}}

In this section, we present the proof of Theorem \ref{thm1.8C}.
Throughout this section, let $r\geq3$
be a fixed integer.
Prior to proceeding with the proof, we first introduce a key definition that underpins the subsequent arguments.

\begin{definition}\label{def5.1}
Let $\varepsilon\in(0,\frac15)$ be a constant, and let
$G$ be a graph with $e(G)>0$.
A subgraph $G'$ with $e(G')>0$ is called an {\bf \emph{$\varepsilon$-dense subgraph}} of $G$
if $G'$ is induced by an edge subset of $G$ such that
$$1\leq e(G)-e(G')\leq \sqrt{2e(G)\big/\big(1\!-\!\frac1r\big)}~~
\text{and}~~\Phi(G')\!-\!\Phi(G)\geq\frac{\varepsilon\big(e(G)\!-\!e(G')\big)}{e(G)}.$$
A graph with at least one edge and no isolated
vertices is called an {\bf \emph{$\varepsilon$-core}}
if it contains no proper $\varepsilon$-dense
subgraphs.
\end{definition}

In the proof of Theorem \ref{thm1.8C},
the main idea is to reduce a counterexample $G$ to an
edge-induced subgraph $H$ that retains almost all edges of $G$ and
contains no proper $\varepsilon_0$-dense subgraph.
The subgraph $H$ possesses several desirable properties that may not hold for $G$.
As a result, we are able to show that $H$ is very close to a complete $r$-partite graph.

\begin{proof}[\textbf{Proof of Theorem \ref{thm1.8C}}]
Let $F$ be a color-critical graph
of order $f$ with $\chi(F)=r+1\geq 4$.
Fix $\varepsilon>0$. Since proving the assertion for a smaller
positive value of $\varepsilon$ implies the assertion for the original
one, we may assume that $\varepsilon$ is sufficiently small in terms
of $F$. Set $C_F:=(\frac{2r}{r-1})^{{(f-2)}/{2}}\alpha_F$ and $\varepsilon_3:=\varepsilon/(2C_F)$.
For convenience, we fix the hierarchy
\begin{equation}\label{align-020F}
0<\varepsilon_0\ll\varepsilon_1\ll\varepsilon_2
\ll\varepsilon_3\ll1.
\end{equation}

Observe first that if $G$ is a regular complete $r$-partite graph,
then $\rho^2(G)=2(1-1/r)e(G)$,
while $N_F(G)=0$ as $\chi(F)=r+1$.
Thus, regular complete $r$-partite graphs form precisely the
exceptional family appearing in the statement of the theorem.
It therefore remains to prove that, for all sufficiently large $m$, every
$m$-edge graph $G$ with no isolated vertices satisfying
$\rho(G)\geq \sqrt{(1-1/r)2m}$
and not being a regular complete $r$-partite graph satisfies
\begin{equation}\label{align-018H}
N_F(G)\geq (C_F-\varepsilon)m^{\frac{f-2}{2}}=\bigl(1-2\varepsilon_3\bigr)C_Fm^{\frac{f-2}{2}}.
\end{equation}

First, we demonstrate that the leading constant in \eqref{align-018H}
is best possible. To this end, it suffices to construct an infinite
sequence of admissible graphs attaining the same asymptotic leading term.
Let $n^\star\to\infty$ through integers divisible by $r$, and set
$m^\star=e(T_{n^\star,r})+1$. We shall construct an $m^\star$-edge graph $G^\star$
satisfying
$$\rho^2(G^\star)\geq \big(1-\frac{1}{r}\big)2m^\star~~\text{and}~~
N_F(G^\star)
=\big(C_F\!+\!o(1)\big)(m^\star)^{\frac{f-2}{2}}.$$

Let $G^\star$ be a graph obtained from $T_{n^\star,\,r}$ by adding a single edge within one of its color classes.
Clearly, $e(G^\star)=m^\star$.
Since $r\mid n^\star$, we have
$m^\star=(1-\frac 1r)\frac{(n^\star)^2}{2}+1$.
Let $\bm{x}=\frac{1}{\sqrt{n^\star}}(1,\ldots,1)^{T}$.
By the Rayleigh quotient, we obtain
$\rho(G^\star) \geq \bm{x}^{T}A(G^\star)\bm{x}
=\frac{2m^\star}{n^\star}.$
Consequently,
$\rho^2(G^\star)\geq (\frac{2m^\star}{n^\star})^2 >2(1-\frac1r)m^\star.$
Thus $G^\star$ satisfies the required spectral condition.
Since $r\mid n^\star$, all color classes of $T_{n^\star,r}$
have the same size, and hence
$N_F(G^\star)=c(n^\star,F)$.
By Lemma \ref{lem2.2C},
$c(n^\star,F)
=\alpha_F(n^\star)^{f-2}+O_F((n^\star)^{f-3})$.
Since $(n^\star)^2=\frac{2r}{r-1}(m^\star-1)$,
it follows that
$N_F(G^\star)
=C_F(m^\star)^{{(f-2)}/{2}}+O_F\big((m^\star)^{{(f-3)}/{2}}\big)$.
Therefore, the leading constant $C_F$ is asymptotically best possible.

Next, suppose, for a contradiction, that the
desired bound in \eqref{align-018H} fails.
Then there exist arbitrarily large integers
$m$ and corresponding $m$-edge graphs $G$
with no isolated vertices, which are not
regular complete $r$-partite graphs, such that
$\rho^2(G)\geq (1-1/r)2m$
and
\begin{align}\label{align-019F}
N_F(G)<\bigl(1-2\varepsilon_3\bigr)C_Fm^{\frac{f-2}{2}}.
\end{align}
Fix such an $m$ sufficiently large with respect to the above
hierarchy \eqref{align-020F}, and fix a corresponding graph $G$.
For readability, we split the proof by contradiction into five steps.

\medskip

\textbf{Step 1: Reduction to an $\varepsilon_0$-core.}
In this step, we iteratively remove light edges to increase
the normalized spectral radius. We show that only
a small proportion of the edges can be removed
from a counterexample. Consequently, the process
produces an $\varepsilon_0$-core that retains
almost all edges of the original graph
(see Claim~\ref{claim5.1C}).

We proceed by constructing a sequence of graphs $G_{(1)}\supset G_{(2)}\supset\cdots\supset G_{(\ell)}$,
where $G_{(1)}=G$ and $G_{(i+1)}$ is an $\varepsilon_{0}$-dense subgraph of $G_{(i)}$ for $1\leq i\leq \ell-1$.
The process stops as soon as either $e(G_{(1)})-e(G_{(\ell)})\geq\lfloor\varepsilon_{0} m\rfloor$
or $G_{(\ell)}$ is an $\varepsilon_{0}$-core.

\begin{claim}\label{claim5.1C}
We have $e(G_{(1)})-e(G_{(\ell)})<\lfloor\varepsilon_{0} m\rfloor$,
meaning that $G_{(\ell)}$ is an $\varepsilon_{0}$-core.
\end{claim}

\begin{proof}
Assume, for contradiction, that \(e(G_{(1)})-e(G_{(\ell)})\geq\lfloor\varepsilon_{0} m\rfloor\).
For every $i\in[\ell-1]$, since $G_{(i+1)}$ is an $\varepsilon_{0}$-dense subgraph of $G_{(i)}$, we have
\[
\Phi(G_{(i+1)})\!-\!\Phi(G_{(i)})
\geq\frac{\varepsilon_{0}\big(e(G_{(i)})\!-\!e(G_{(i+1)})\big)}{e(G_{(i)})}\geq\frac{\varepsilon_{0}}{m}\big(e(G_{(i)})\!-\!e(G_{(i+1)})\big).
\]
Combining this with the assumption that $e(G_{(1)})-e(G_{(\ell)})\geq\lfloor\varepsilon_{0} m\rfloor$,
we can deduce that
\begin{align*}
\Phi(G_{(\ell)})
\!-\!\Phi(G_{(1)})
\geq\frac{\varepsilon_{0}}m\sum_{i=1}^{\ell-1}\big(e(G_{(i)})\!-\!e(G_{(i+1)})\big)
=\frac{\varepsilon_{0}}{m}\big(e(G_{(1)})\!-\!e(G_{(\ell)})\big)
>\frac{1}{2}{\varepsilon_{0}}^2.
\end{align*}
Since $\rho(G_{(1)})=\rho(G)\geq\sqrt{(1-\frac{1}{r})2m}$, we have
$\Phi(G_{(1)})\geq\sqrt{2(1\!-\!\frac{1}{r})}$.
It follows that
\begin{align}\label{align-021F}
\Phi(G_{(\ell)})
\geq\sqrt{2\big(1\!-\!\frac{1}{r}\big)}\!+\!\frac{1}{2}{\varepsilon_{0}}^2
\geq\sqrt{2\big(1\!-\!\frac{1}{r}\!+\!\frac{1}{2}{\varepsilon_{0}}^2\big)}.
\end{align}

On the other hand, since $G_{(\ell-1)}$ is not the final graph in the sequence, we know that
$e(G_{(1)})-e(G_{(\ell-1)})<\lfloor\varepsilon_{0} m\rfloor$.
Additionally, since $G_{(\ell)}$ is an $\varepsilon_0$-dense subgraph of $G_{(\ell-1)}$,
we have
$$e(G_{(\ell-1)})\!-\!e(G_{(\ell)})\leq \sqrt{2 e(G_{(\ell-1)})\big/\big(1\!-\!\frac1r\big)\,}
< \lfloor\varepsilon_{0} m\rfloor.$$
Hence, $e(G_{(1)})-e(G_{(\ell)})<2\lfloor\varepsilon_{0} m\rfloor$, and thus $e(G_{(\ell)})>\frac m2$.
Hence, by \eqref{align-019F},
$$N_F(G_{(\ell)})\leq N_F(G)<
(1-2\varepsilon_3)C_Fm^{\frac{f-2}{2}}
\leq 2^{\frac{f-2}{2}}C_F
e(G_{(\ell)})^{\frac{f-2}{2}}.$$
On the other hand, \eqref{align-021F} gives
$\rho(G_{(\ell)}) \geq \sqrt{2(1-\frac{1}{r}+\frac{1}{2}{\varepsilon_0^2})e(G_{(\ell)})}$.
Applying Theorem \ref{thm1.3C} with
$\varepsilon$ taken to be $\frac12\varepsilon_0^2$, there exists a constant
$\delta>0$ such that, for sufficiently large
$e(G_{(\ell)})$, we have
$N_F(G_{(\ell)}) \geq \delta\, e(G_{(\ell)})^{f/2}$.
Since $e(G_{(\ell)})>m/2$ and $m$ is sufficiently large, we have
$2^{\frac{f-2}{2}}C_Fe(G_{(\ell)})^{\frac{f-2}{2}}<\delta\,e(G_{(\ell)})^{f/2}$,
a contradiction.
Therefore,
$e(G_{(1)})-e(G_{(\ell)})<\lfloor\varepsilon_0m\rfloor$.

Consequently, the first stopping condition cannot occur. Hence the
iterative procedure must terminate because $G_{(\ell)}$ contains no
proper $\varepsilon_0$-dense subgraph. Moreover, by the definition of
an edge-induced subgraph, $G_{(\ell)}$ has no isolated vertices.
Therefore, $G_{(\ell)}$ is an $\varepsilon_0$-core.
\end{proof}

\textbf{Step 2: Spectral stability of the $\varepsilon_0$-core.}
Set $H:=G_{(\ell)}$ and $h:=e(H)$.
By Claim \ref{claim5.1C}, $H$ is an $\varepsilon_0$-core and
$m-h<\lfloor\varepsilon_0 m\rfloor$.
The core property alone does not determine the exact structure of $H$.
We first show that $H$ contains no light edges (see Claim \ref{claim5.2B}).
Furthermore, the spectral condition is preserved under the reduction procedure,
which enables us to invoke the edge-spectral stability theorem and deduce that $H$ admits an almost Tur\'an structure
(see Claims \ref{claim5.2C}, \ref{claim5.3C} and \ref{claim5.4C}).

\begin{claim}\label{claim5.2B}
With respect to every non-negative unit eigenvector corresponding
to $\rho(H)$, the $\varepsilon_0$-core $H$ contains no light edges.
\end{claim}

\begin{proof}
Let $\bm{x}=(x_v)_{v\in V(H)}$ be an arbitrary non-negative unit
eigenvector corresponding to $\rho(H)$.
Suppose to the contrary that $uv$ is a light edge of $H$
with respect to $\bm{x}$.
Since $G_{(i+1)}$ is an $\varepsilon_0$-dense subgraph of $G_{(i)}$ for every $i\in[\ell-1]$, Definition \ref{def5.1} gives
$\Phi(G_{(i+1)})-\Phi(G_{(i)})>0.$
Therefore,
$\Phi(H)\geq \Phi(G_{(1)})\geq \sqrt{2(1-1/r)}.$
Equivalently,
$\rho(H)\geq\sqrt{2(1-1/r)e(G_{(\ell)})}.$

Let $H'$ denote the subgraph induced by
$E(H)\setminus\{uv\}$. Then,
$e(H)-e(H')=1.$
Set $a:=2(1-1/r).$
Since $r\geq3$, we have
$\frac{4}{3}\leq a<2.$
By the preceding inequality, we have
$\rho(H)\geq\sqrt{a\,e(H)}$.
Since $h$ is sufficiently large, we have
$\lfloor\varepsilon_0 h\rfloor\geq2$.
Apply the procedure in Definition~\ref{def1.1}
to $H$ with parameter $\varepsilon_0$,
taking $\bm{x}^{(1)}=\bm{x}$ and
$u_1v_1=uv$.
Then by Lemma~\ref{lem2.8C} (ii), we obtain
$$
\Phi(H')-\Phi(H)
\geq\frac{1}{5e(H)}
>\frac{\varepsilon_0}{e(H)}.
$$
Since $e(H)-e(H')=1$, we further obtain
$\Phi(H')-\Phi(H)>\varepsilon_0(e(H)-e(H'))\big/e(H).$
Furthermore, as $m$ is sufficiently large and $e(H)>m/2$, we have
$1=e(H)-e(H')\leq \sqrt{{2e(H)}/{(1-1/r)}}.$
Therefore, by Definition \ref{def5.1}, $H'$ is an $\varepsilon_0$-dense subgraph of $H$.

On the other hand, we have already proved that
$e(G_{(1)})-e(H)<\lfloor\varepsilon_0m\rfloor.$
Thus, according to the stopping rule in the construction of the sequence
$G_{(1)}\supset G_{(2)}\supset\cdots\supset G_{(\ell)}$,
the graph $H$ contains no $\varepsilon_0$-dense subgraph.
This contradicts the fact that $H'$ is an $\varepsilon_0$-dense subgraph of $H$.
Therefore, $H$ contains no light edges.
\end{proof}

\begin{claim}\label{claim5.2C}
We have $\rho^2(H)\geq (1\!-\!\frac{1}{r})2h$ for $\ell=1$, and
$\rho^2(H)>(1\!-\!\frac{1}{r})2h$ for $\ell\geq 2$. Moreover,
there exists a Tur\'{a}n graph $T_{n',r}$ such that $V(T_{n',r})\subseteq V(H)$ and  $d(H,T_{n',r})\leq\varepsilon_{0} h$.
\end{claim}

\begin{proof}
Claim \ref{claim5.1C} gives
$m-h<\varepsilon_{0} m$,
which together with $\varepsilon_{0} \ll1$ implies that $\frac12m<h\leq m$.
Combining this with \eqref{align-019F}, we obtain
\begin{align}\label{align-022F}
N_F(H)
\leq N_F(G)
<2^{\frac{f-2}{2}}C_Fh^{\frac{f-2}{2}}.
\end{align}
Recall that $\Phi(G_{(1)})\geq \sqrt{2(1-1/r)}$.
If $\ell=1$, then $H=G_{(1)}$ and $\rho^2(H)=\rho^2(G_{(1)})\geq(1\!-\!{1}/{r})2h.$
If $\ell\geq 2$, then by the definition of \(\varepsilon_{0}\)-dense subgraphs, we know that
\begin{align*}
\Phi(H)>\Phi(G_{(\ell-1)}) >\cdots >\Phi(G_{(1)})\geq\sqrt{2\big(1\!-\!\frac{1}{r}\big)}.
 \end{align*}
Thus, $\rho^2(H)>(1-1/r)2h$.
Apply Theorem \ref{thm1.6C} (i) with its parameter
$\varepsilon$ taken to be $\varepsilon_0$, and let
$\eta,\delta>0$ and $m_0$ be the corresponding constants.
Since $r\geq 3$ and
$\rho^2(H)\geq 2(1-\frac 1r)h>2(1-\frac 1r-\delta)h$,
the spectral condition in Theorem \ref{thm1.6C} (i) is satisfied.
Moreover, since $h>m/2$ and $m$ is sufficiently large, we may
assume that $h\geq m_0$. By \eqref{align-022F},
$N_F(H)\leq 2^{\frac{f-2}{2}}C_Fh^{\frac{f-2}{2}}$.
Since $\eta>0$ is fixed and $h$ is sufficiently large, we further have
$2^{\frac{f-2}{2}}C_Fh^{\frac{f-2}{2}} \leq \eta h^{f/2}$.
Hence $N_F(H)\leq \eta h^{f/2}$.
Thus, Theorem \ref{thm1.6C} (i) applies to $H$.
Consequently, there exists a Tur\'{a}n graph $T_{n',r}$
such that $V(T_{n',r})\subseteq V(H)$ and $d(H,T_{n',r})\leq\varepsilon_0 h$.
This proves the claim.
\end{proof}

For convenience, we write $n:=|H|$ and $\psi:=\sqrt{2h/(1\!-\!\frac1r)}$.

\begin{claim}\label{claim5.3C}
Let $n'$ be defined as in Claim \ref{claim5.2C}. Then we have
$$\big(1\!-\!\varepsilon_{1}\big)\psi
\leq n'\leq n
\leq\big(1\!+\!{\varepsilon_{1}}^2\big)n'
\leq \big(1\!+\!\varepsilon_{1}\big)\psi.$$
\end{claim}

\begin{proof}
By Claim \ref{claim5.2C}, we have $d(H,T_{n',r})\leq \varepsilon_{0} h$,
which implies that $n'$ is sufficiently large and $(1-\varepsilon_{0}) h\leq e(T_{n',r})\leq (1+\varepsilon_{0}) h$.
Furthermore, it is known that $\frac{r-1}{2r}{n'}^2-\frac{r}{8}\leq e(T_{n',r})\leq \frac{r-1}{2r}{n'}^2$, which implies that
\begin{align*}
2e(T_{n',r})\big/\big(1\!-\!\frac1r\big)
\leq {n'}^2\leq\big(1\!+\!\varepsilon_{0}\big)2e(T_{n',r})\big/\big(1\!-\!\frac1r\big),
\end{align*}
Combining the two inequalities derived above, we can deduce that
\begin{align}\label{align-023F}
\sqrt{\big(1\!-\!\varepsilon_{0}\big)}\psi
\leq n'\leq \big(1\!+\!\varepsilon_{0}\big)\psi.
\end{align}

By Claim \ref{claim5.2B}, $H$ admits no light edges,
and by Claim \ref{claim5.2C}, we know that $\rho^2(H)\geq \frac{4}{3}h$ since $r\geq 3$.
Setting $a=\frac{4}{3}$ and applying (iii) of Lemma \ref{lem2.8C} to $H$, we derive that
$\delta(H)\geq\sqrt{h}\big/24^9$.
Define $R:= V(H)\setminus V(T_{n',r})$ (which may be empty).
Then, we have
$$d\big(H,T_{n',r}\big)\geq e\big(R\big)\!+\!e\big(R, V(T_{n',r})\big)\geq\frac12\sum_{v\in R}d_{H}(v)\geq\frac1{2\times24^9}|R|\sqrt{h}.$$

Recall that $d(H,T_{n',r})\leq \varepsilon_{0} h$.
We further derive
$|R|\leq 2\times24^9\varepsilon_{0}\sqrt{h}.$
Combining this with the left-hand side of \eqref{align-023F} and noting that $\varepsilon_{0}\ll \varepsilon_{1}$,
we obtain $|R|\leq {\varepsilon_{1}}^2 n'$ and $n=n'+|R|\leq (1+{\varepsilon_{1}}^2) n'$.
Again applying \eqref{align-023F}, we immediately obtain
\begin{align*}
\big(1\!-\!\varepsilon_{1}\big)\psi
\leq \big(1\!-\!\varepsilon_{0}\big)\psi
\leq n'\leq n\leq\big(1+{\varepsilon_{1}}^2\big)n'\leq\big(1+{\varepsilon_{1}}^2\big)\big(1+\varepsilon_{0}\big)\psi\leq\big(1+\varepsilon_{1}\big)\psi.
\end{align*}
This proves the claim.
\end{proof}

Let $\cup_{i=1}^{r}V_i$ be a partition of $V(H)$
such that {\bf $\sum_{1\leq i<j\leq r}e(V_i,V_j)$ is maximized}.

\begin{claim}\label{claim5.4C}
We have
$\sum_{i\in[r]}e(H[V_i])\leq {\varepsilon_{1}}^2 n^2$ and $\big||V_i|-\frac {n}r\big|\leq 3\varepsilon_{1} n$
for each $i\in [r]$.
\end{claim}

\begin{proof}
Let $U_1,U_2,\dots,U_r$ be the color classes of $T_{n',r}$.
Recall that $R= V(H)\setminus V(T_{n',r})$.
Define $U_1^*=U_1\cup R$ and $U_i^*=U_i$ for $2\leq i\leq r$.
Since every $U_i$ is an independent set of $T_{n',r}$,
it follows that
$d(H,T_{n',r})\geq\sum_{i\in[r]}e(H[U_i^*])$.
Recall that $d(H,T_{n',r})\leq \varepsilon_{0} h$.
We further deduce that
\begin{align}\label{align-026H}
\sum_{i\in[r]}e\big(H[U_i^*]\big)\leq d\big(H,T_{n',r}\big)\leq \varepsilon_{0} h\leq {\varepsilon_{1}}^2 n^2,
\end{align}
where the last inequality follows from $\varepsilon_{0} \ll \varepsilon_{1}$ and $h\leq \binom{n}{2}$.

Observe that $\cup_{i\in[r]}U_i^*$ is a partition of $V(H)$,
and recall that $\cup_{i\in[r]}V_i$ is also a partition of $V(H)$
such that $\sum_{1\leq i<j\leq r}e(V_i,V_j)$ is maximized.
Thus,
$$\sum_{i\in[r]}e\big(H[V_i]\big)\leq \sum_{i\in[r]}e\big(H[U_i^*]\big)\leq{\varepsilon_{1}}^2 n^2.$$

Define $\alpha=\max\big\{\big||V_i|-\frac {n}r\big|: i\in [r]\big\}$,
and assume without loss of generality that $\alpha=\big||V_1|-\frac {n}r\big|$.
By the Cauchy-Schwarz inequality, we obtain
$\big(\sum_{i=2}^{r}|V_i|\big)^2\leq(r-1)\sum_{i=2}^{r}|V_i|^2$.
Hence,
\begin{equation*}
\sum_{2\leq i<j\leq r}\!\!\!2|V_i||V_j|=
\Big(\sum_{i=2}^{r}|V_i|\Big)^2\!-\!\sum_{i=2}^{r}|V_i|^2
\leq\frac{r\!-\!2}{r\!-\!1}\Big(\sum_{i=2}^{r}|V_i|\Big)^2=\frac{r\!-\!2}{r\!-\!1}\big(n\!-\!|V_1|\big)^2.
\end{equation*}
Consequently, we can deduce that
\begin{align}\label{align-025F}
e(H)
&\leq \sum_{1\leq i<j\leq r}\!\!\!\!|V_i||V_j|+\sum_{i\in[r]}\!e\big(H[V_i]\big)
\leq |V_1|(n\!-\!|V_1|)\!+\!\!\sum_{2\leq i<j\leq r}\!\!\!\!|V_i||V_j|
\!+\!{\varepsilon_{1}}^2 n^2 \nonumber\\
&\leq |V_1|(n\!-\!|V_1|)\!+\!\frac{r\!-\!2}{2(r\!-\!1)}\big(n\!-\!|V_1|\big)^2\!+\!{\varepsilon_{1}}^2 n^2 \nonumber\\
&=-\frac{r}{2(r\!-\!1)}\alpha^2\!+\!\frac{r\!-\!1}{2r}n^2\!+\!{\varepsilon_{1}}^2 n^2,
\end{align}
where the last equality holds as $\alpha=\big||V_1|-\frac{n}r\big|$.

From Claim \ref{claim5.3C}, we know that $n'\leq n\leq(1\!+\!{\varepsilon_{1}}^2)n'$,
which implies that
$e(T_{n,r})\!-\!e(T_{n',r})\leq (n\!-\!n')n\leq{\varepsilon_{1}}^2 n'n\leq {\varepsilon_{1}}^2 n^2.$
Thus, we derive $e(T_{n',r})\geq \frac{r-1}{2r}n^2-2{\varepsilon_{1}}^2 n^2$, since $e(T_{n,r})\geq \frac{r-1}{2r}{n}^2-\frac{r}{8}$.
In light of \eqref{align-026H}, we deduce that
\begin{align}\label{align-026F}
e(H)\geq e(T_{n',r})\!-\!{\varepsilon_{1}}^2 n^2\geq\frac{r\!-\!1}{2r}n^2\!-\!3{\varepsilon_{1}}^2 n^2.
\end{align}
Combining this with \eqref{align-025F} gives
$\frac{r}{2(r-1)}\alpha^2\leq 4{\varepsilon_{1}}^2n^2$.
Hence, $\alpha<\sqrt{8}\varepsilon_{1} n<3\varepsilon_{1} n$, as claimed.
\end{proof}

We have shown that the $\varepsilon_0$-core $H$ is close
to a Tur\'an graph. However, this approximate structure
does not rule out vertices with small total degree
or large internal degree within their partition classes.
We next show that neither type of vertex can occur.

\medskip

\noindent
\textbf{Step 3: Elimination of irregular vertices.}
For each $i\in[r]$, let
$W_i:=\{v\in V_i:d_{V_i}(v)\geq4\varepsilon_1n\}$.
For $k\in\{1,2\}$, let
$S^{(k)}:=\{v\in V(H):d_H(v)\leq (1-1/r-4\varepsilon_k)n\}$.
The argument proceeds in three stages.
First, we prove that all these
irregular sets are small (see Claims \ref{claim5.5C} and \ref{claim5.6C}).
Next, using common-neighborhood estimates and
a supersaturation argument,
we show that $\cup_{i\in[r]}W_i\subseteq S^{(2)}\subseteq S^{(1)}$
(see Claims \ref{claim5.7C}--\ref{claim5.10C}).
Finally,
we prove that the Perron coordinates of the remaining regular vertices
are nearly equal and use the $\varepsilon_0$-core property to eliminate
$S^{(2)}$ entirely (see Claims \ref{claim5.11C}, \ref{claim5.12C} and \ref{claim5.13C}).

\begin{claim}\label{claim5.5C}
We have $|S^{(k)}|\leq \varepsilon_{1} n$ for $k\in \{1,2\}$.
\end{claim}

\begin{proof}
By way of contradiction, assume that $|S^{(k)}|>\varepsilon_{1} n$.
Then, $S^{(k)}$ must contain a subset $S'$ of size $\lfloor\varepsilon_{1} n\rfloor$.
Denote by $H'$ the subgraph of $H$ induced by $V(H)\setminus S'$.
For $k\in\{1,2\}$ and any vertex $v\in S'$,
we have $d_{H}(v)\leq \big(1-\frac{1}{r}-4\varepsilon_{1} \big)n$ since $\varepsilon_1\ll \varepsilon_2$.
In light of \eqref{align-026F}, we have
\begin{align*}
e(H')&\geq e(H)\!-\!\sum_{v\in S'}d_{H}(v)
\geq \frac{r\!-\!1}{2r}n^2\!-\!3{\varepsilon_{1}}^2 n^2\!-\!\varepsilon_{1}n\big(1\!-\!\frac{1}{r}\!-\!4\varepsilon_{1}\big)n\nonumber\\
&=\big(\frac{r\!-\!1}{2r}\!-\!\frac{r\!-\!1}{r}\varepsilon_{1}\!+\!{\varepsilon_{1}}^{2}\big)n^2
\geq\big(\frac{r\!-\!1}{2r}\!+\!\frac12{\varepsilon_{1}}^2\big){\big|H'\big|}^2,
\end{align*}
where the last inequality follows from $|H'|=n\!-\!|S'|=\lceil(1\!-\!\varepsilon_1)n\rceil$
and sufficiently large $n$.
Consequently,
$e(H')\geq (1-1/r+\varepsilon_1^2)\frac{|H'|^2}{2}$.
Since $|H'|=\Theta(n)$, the graph $H'$ is sufficiently large whenever $n$ is sufficiently large. Therefore, applying Theorem \ref{lem-ES-super} with $\varepsilon=\varepsilon_1^2$, we obtain
\begin{align}\label{align-027F}
N_F(H)
\geq N_F(H')
\geq \delta |H'|^f,
\end{align}
where $\delta=\delta(F,\varepsilon_1^2)>0$ is the constant provided by Theorem~\ref{lem-ES-super}.
On the other hand,
Claim \ref{claim5.3C} gives
$h=\Theta(n^2)=\Theta(|H'|^2)$.
Together with \eqref{align-022F}, this yields
$N_F(H)=O_F\bigl(h^{{(f-2)}/{2}}\bigr)=O_F\bigl(|H'|^{f-2}\bigr)$.
Since $\delta>0$ is fixed and $|H'|$ is sufficiently large, this
contradicts \eqref{align-027F}.
Hence, $|S^{(k)}|\leq \varepsilon_{1} n$, as claimed.
\end{proof}

\begin{claim}\label{claim5.6C}
We have $\sum_{i\in[r]}|W_i|\leq \varepsilon_{1} n$.
\end{claim}

\begin{proof}
For each $i\in [r]$, we see that
\begin{equation*}
e\big(H[V_i]\big)=\sum\limits_{v\in V_i}\frac12d_{V_i}(v)\geq
\sum\limits_{v\in W_i}\frac12d_{V_i}(v)\geq\varepsilon_{1} n|W_i|.
\end{equation*}
Summing over all \(i\in[r]\), we obtain
$\sum_{i\in[r]}e(H[V_i])\geq\varepsilon_{1} n\sum_{i\in[r]}|W_i|.$
On the other hand, Claim \ref{claim5.4C} gives
$\sum_{i\in[r]}e(H[V_i])\leq {\varepsilon_{1}}^2 n^2$.
It follows that $\sum_{i\in[r]}|W_i|\leq\varepsilon_{1} n$, as desired.
\end{proof}

\begin{claim}\label{claim5.7C}
For each $k\in\{1,2\}$ and $i\in[r]$, let
$V_i^{(k)}:=V_i\setminus(W_i\cup S^{(k)})$.
Fix $k\in\{1,2\}$ and $i_0\in[r]$.
For any $1\leq s\leq f$ and
$u_1,\ldots,u_s\in
\bigcup_{i\in[r]\setminus\{i_0\}}V_i^{(k)}$,
the following hold:

\vspace{1mm}
{\rm (i)}
The vertices $u_1,\ldots,u_s$ have at least
$(\frac1r-3f^2\varepsilon_2)n$ common neighbors in
$V_{i_0}^{(k)}$.

\vspace{0.5mm}
{\rm (ii)}
If
$u_0\in \bigcup_{i\in[r]\setminus\{i_0\}}(W_i\setminus S^{(k)}),$
then $u_0,u_1,\ldots,u_s$ have at least
$\frac{n}{2r^2}$ common neighbors in $V_{i_0}^{(k)}$.
\end{claim}

\begin{proof}
{\rm (i)}
For each $j\in[s]$, suppose that
$u_j\in V_{i_j}^{(k)}$ for some
$i_j\in[r]\setminus\{i_0\}$.
Since
$V_{i_j}^{(k)}=V_{i_j}\setminus(W_{i_j}\cup S^{(k)})$,
the definitions of $W_{i_j}$ and $S^{(k)}$ imply that
$d_{V_{i_j}}(u_j)<4\varepsilon_1n<2\varepsilon_2n$
and
$d_H(u_j)>(1-\frac1r-4\varepsilon_2)n.$
Also, recall that
$|V_i|\leq(\frac1r+3\varepsilon_1)n
<(\frac1r+2\varepsilon_2)n$
for every $i\in[r]$.
Thus,
\begin{align}\label{align-034F}
d_{V_{i_0}}(u_j)
&=d_H(u_j)-d_{V_{i_j}}(u_j)
-\!\!\!\sum_{i\in[r]\setminus\{i_0,i_j\}}\!\!\!d_{V_i}(u_j)>
\big(1-\frac1r-6\varepsilon_2\big)n
-(r-2)\big(\frac1r+2\varepsilon_2\big)n\nonumber\\
&=\big(\frac1r-2(r+1)\varepsilon_2\big)n
\geq\big(\frac1r-2f\varepsilon_2\big)n,
\end{align}
where the last inequality follows from
$f\geq\chi(F)=r+1$.
It follows from \eqref{align-034F} that
\begin{align*}
\Big|\bigcap_{j=1}^sN_{V_{i_0}}(u_j)\Big|
&\geq
\sum_{j=1}^sd_{V_{i_0}}(u_j)
-(s-1)|V_{i_0}|
>
s\big(\frac1r-2f\varepsilon_2\big)n
-(s-1)\big(\frac1r+2\varepsilon_2\big)n\\
&=
\big(
\frac1r-
(2sf+2s-2)\varepsilon_2
\big)n
\geq
\big(
\frac1r-(3f^2-2)\varepsilon_2
\big)n,
\end{align*}
where the last inequality follows from $1\leq s\leq f$.
By Claims \ref{claim5.5C} and \ref{claim5.6C}, we have
$|W_{i_0}\cup S^{(k)}|\leq 2\varepsilon_1n\leq2\varepsilon_2n.$
The definition of $V_{i_0}^{(k)}$ thus gives
\begin{align*}
\Big|\bigcap_{j=1}^s
N_{V_{i_0}^{(k)}}(u_j)\Big|
\geq
\Big|\bigcap_{j=1}^s
N_{V_{i_0}}(u_j)\Big|
-|W_{i_0}\cup S^{(k)}|
\geq
\big(\frac1r-3f^2\varepsilon_2\big)n.
\end{align*}
This proves assertion {\rm (i)}.

\vspace{1mm}
{\rm (ii)}
Suppose that
$u_0\in W_{i'}\setminus S^{(k)}$
for some $i'\in[r]\setminus\{i_0\}$.
By the definition of $S^{(k)}$, we know that
$$
d_H(u_0)>
\big(1-\frac1r-4\varepsilon_k\big)n
\geq
\big(1-\frac1r-4\varepsilon_2\big)n.
$$
Since $\bigcup_{i\in[r]}V_i$ is a partition of $V(H)$
for which
$\sum_{1\leq i<j\leq r}e(V_i,V_j)$ is maximized,
it follows that
$d_{V_{i'}}(u_0)\leq\frac1r d_H(u_0).$
Recall also that
$|V_i|\leq(\frac1r+2\varepsilon_2)n$
for every $i\in[r]$.
Therefore,
\begin{align}\label{align-028F}
d_{V_{i_0}}(u_0)
&=d_H(u_0)-d_{V_{i'}}(u_0)
-\!\!\!\sum_{i\in[r]\setminus\{i_0,i'\}}\!\!\!d_{V_i}(u_0)
\geq
\big(1-\frac1r\big)d_H(u_0)
-(r-2)\big(\frac1r+2\varepsilon_2\big)n\nonumber\\
&>
\big(1-\frac1r\big)
\big(1-\frac1r-4\varepsilon_2\big)n
-(r-2)\big(\frac1r+2\varepsilon_2\big)n
>
\big(\frac1{r^2}-2f\varepsilon_2\big)n,
\end{align}
where the last inequality follows from
$4(1-\frac1r)+2(r-2)<2r<2f$.

Let $A:=\bigcap_{j=1}^sN_{V_{i_0}^{(k)}}(u_j).$
By assertion {\rm (i)}, we have
$|A|\geq(\frac1r-3f^2\varepsilon_2)n.$
Since $A\subseteq V_{i_0}^{(k)}\subseteq V_{i_0}$,
combining this bound with \eqref{align-028F} and the estimate $|V_{i_0}|\leq(\frac1r+2\varepsilon_2)n$ yields
\begin{align*}
\Big|N_{V_{i_0}^{(k)}}(u_0)\cap A\Big|
=\Big|N_{V_{i_0}}(u_0)\cap A\Big|
\geq d_{V_{i_0}}(u_0)+|A|-|V_{i_0}|
\geq\Big(\frac1{r^2}-\big(3f^2+2f+2\big)\varepsilon_2\Big)n
\geq\frac{n}{2r^2}.
\end{align*}
Hence, $u_0,u_1,\ldots,u_s$ have at least
$\frac{n}{2r^2}$ common neighbors in
$V_{i_0}^{(k)}$, as required.
\end{proof}

For every edge $e\in \cup_{i\in [r]}E(H[V_i])$,
let $N_F(H,e)$ denote the number of copies of $F$ in $H$ that contain only the edge $e$ from $ \cup_{i\in[r]}E(H[V_i])$.
Since $F$ is a color-critical graph with $|F|=f$ and $\chi(F)=r+1$,
there exists a good edge $xy\in E(F)$ such that $F\!-\!xy$ is $r$-partite.
We may assume that $L_1,L_2,\ldots,L_r$ are the color classes of a
proper $r$-coloring of $F-xy$, with $x,y\in L_1$.

\begin{claim}\label{claim5.8C}
If $uv$ is an edge in $E(H[V_i])$ for some $i\in [r]$, with $u\in V_i\setminus S^{(2)}$ and $v\in V^{(2)}_i$,
then there exists a constant $\gamma>0$ such that $N_F(H,uv)\geq \gamma \cdot c(n,F)$.
\end{claim}

\begin{proof}
Assume, without loss of generality, that $uv\in E(H[V_1])$.
Note that $|L_1|<f$. By Claim \ref{claim5.4C}, it follows that $|V_1|\geq(\frac1r-3\varepsilon_{1})n>f>|L_1|.$
First, we choose $|L_1|-2$ vertices from $V_{1}^{(2)}\setminus\{u,v\}$
and combine them with $\{u,v\}$ to form an $|L_1|$-subset $\widetilde{L}_1$ of $V_1$.

Since $u\in V_1\setminus S^{(2)}$, it follows that either $u\in W_1\setminus S^{(2)}$ or $u\in V^{(2)}_1$.
Moreover, $\widetilde{L}_1\setminus \{u\}\subseteq V^{(2)}_1$.
Given that $|\widetilde{L}_1|<f$,
Claim \ref{claim5.7C} indicates that the vertices in $\widetilde{L}_1$ have at least $\frac{n}{2r^2}$ common neighbors in $V_2^{(2)}$.
Therefore, we can select an $|L_2|$-subset \(\widetilde{L}_2\subseteq V_2^{(2)}\) in which every vertex is adjacent to all vertices in $\widetilde{L}_1$.
By iteratively applying Claim \ref{claim5.7C} for $j\in \{3,\ldots,r\}$,
we conclude that the vertices in $\cup_{i\in [j\!-\!1]}\widetilde{L}_i$ have at least $\frac{n}{2r^2}$ common neighbors in $V_j^{(2)}$.
Hence, we can choose an $|L_j|$-subset $\widetilde{L}_j\subseteq V_j^{(2)}$
in which every vertex is adjacent to all vertices in $\cup_{i\in[j\!-\!1]}\widetilde{L}_i$.

Let $F^*$ be the subgraph of $H$ obtained from $K_{\widetilde{L}_1,\ldots,\widetilde{L}_r}$ by adding the edge $uv$.
It is clear that
$$N_{F^*}(H,uv)\geq \binom{|V_1^{(2)}|\!-\!2}{|L_1|\!-\!2}\prod_{j=2}^{r}
\binom{\big\lfloor n/(2r^2)\big\rfloor}{|L_j|}.$$

Recall that $|V_1|\geq(\frac1r-3\varepsilon_{1})n$ and $|W_1\cup S^{(2)}|\leq2\varepsilon_{1}n.$
Thus, $|V_1^{(2)}|\ge|V_1|-|W_1\cup S^{(2)}|\geq\frac{n}{2r}$.
Moreover, since $\sum_{i\in [r]}|L_i|\!=\!f$,
it follows that $N_{F^*}(H,uv)=\Omega_F(n^{f-2})$.

Since $F-xy$ has color classes $L_1,\ldots,L_r$ and
$|\widetilde L_i|=|L_i|$ for every $i\in[r]$,
the graph $F$ is a spanning subgraph of $F^*$:
we may map $L_i$ bijectively onto $\widetilde L_i$ for each
$i\in[r]$, with the edge $xy$ mapped onto $uv$.
Moreover, since $uv$ is the unique edge of $F^*$ lying inside
one of the partition classes $\widetilde L_i$, every such copy of
$F$ is counted by $N_F(H,uv)$.

Since $F$ is fixed and $|V(F^*)|=|V(F)|=f$, each fixed copy of
$F$ can be contained in at most $M_F$ copies of $F^*$, where
$M_F>0$ is a constant depending only on $F$.
Therefore,
$N_{F^*}(H,uv)\leq {M_F}N_F(H,uv),$
and hence
$N_F(H,uv)\geq \frac{1}{M_F}N_{F^*}(H,uv)=\Omega_F(n^{f-2}).$
By Lemma \ref{lem2.2C}, we know that $c(n, F)=\Theta(n^{f-2})$.
Therefore, there exists a constant $\gamma= \gamma(F)>0$ such that $N_F(H,uv)\geq \gamma \cdot c(n,F)$, as required.
\end{proof}

\begin{claim}\label{claim5.9C}
We have $N_{F}(H)<(1-\varepsilon_3)c(n,F)$.
\end{claim}

\begin{proof}
By Claim \ref{claim5.1C} and the hierarchy established in \eqref{align-020F},
we have $m-h<\lfloor\varepsilon_{0} m\rfloor<\frac12\varepsilon_{1}m$, which implies that $m\leq (1+\varepsilon_{1})h$.
Moreover, Claim \ref{claim5.3C} indicates $n^2\geq(1-\varepsilon_{1})^2\cdot2h/(1-\frac1r).$
Thus, we derive that
$2m/(1\!-\!\frac1r)\leq (1+\varepsilon_{1})\cdot 2h/(1\!-\!\frac1r)\leq(1+4\varepsilon_{1})n^2.$
It follows that
\begin{align}\label{align-029G}
\Big(2m\big/\big(1\!-\!\frac1r\big)\Big)^{\frac{f-2}{2}}
\leq \big(1+4\varepsilon_{1}\big)^{\frac{f-2}{2}}n^{f-2}
\leq \big(1+\frac12\varepsilon_3\big)n^{f-2},
\end{align}
where the last inequality holds as $\varepsilon_{1}\ll \varepsilon_3$.
By Lemma \ref{lem2.2C}, for sufficiently large $n$, we have
$c(n,F)\geq (1-\frac12\varepsilon_3)\alpha_F n^{f-2}$.
Combining this with inequalities \eqref{align-019F} and \eqref{align-029G}, we deduce that
\begin{align*}
N_F(G)
&<\alpha_F\bigl(1-2\varepsilon_3\bigr)
\Big(2m\big/\big(1\!-\!\frac1r\big)\Big)^{\frac{f-2}{2}}\!\!\leq
\bigl(1-2\varepsilon_3\bigr)\big(1+\frac12\varepsilon_3\big)\alpha_F n^{f-2}\\
&\leq \bigl(1-\frac32\varepsilon_3\bigr)\big(1-\frac12 \varepsilon_3\big)^{-1}c(n,F)\leq (1-\varepsilon_3)c(n,F),
\end{align*}
Consequently, we obtain $N_{F}(H)\leq N_F(G)<(1-\varepsilon_3)c(n,F),$ as desired.
\end{proof}

Recall that $W_i=\{u\in\! V_i: d_{V_i}(u)\geq4\varepsilon_{1}n\}$
and $S^{(2)}=\{u\in\! V(H): d_H(u)\leq(1\!-\!\frac{1}{r}\!-\!4\varepsilon_{2})n\}$.
The next step is to prove that both $\cup_{i\in[r]}W_i$ and  $S^{(2)}$ are empty sets.

\begin{claim}\label{claim5.10C}
We have $\cup_{i\in[r]}W_i\subseteq S^{(2)}\subseteq S^{(1)}$.
\end{claim}

\begin{proof}
By the definitions of $S^{(1)}$ and $S^{(2)}$, we immediately obtain $S^{(2)}\subseteq S^{(1)}$.
It suffices to prove that $\cup_{i\in[r]}W_i\subseteq S^{(2)}$.
Suppose, for the sake of contradiction, that there exists some $i\in [r]$ such that $W_{i}\setminus S^{(2)}\neq \varnothing$.
Choose a vertex $u\in W_{i}\setminus S^{(2)}$.
Based on the definition of $W_i$, $u$ has at least $4\varepsilon_{1} n$ neighbors in $V_{i}$.
By Claims \ref{claim5.5C} and \ref{claim5.6C},
we have $|W_i\cup S^{(2)}|\leq 2\varepsilon_{1} n$.
Since $V_i^{(2)}=V_i\setminus (W_i\cup S^{(2)})$,
it follows that $u$ has at least $2\varepsilon_{1} n$ neighbors in $V_{i}^{(2)}$.

Now, set $\gamma^*=\lceil1/\gamma\rceil$, where $\gamma$ is a constant defined in Claim \ref{claim5.8C}.
Since $\gamma^*<2\varepsilon_{1} n$, we can find $\gamma^*$ neighbors of $u$ in $V_{i}^{(2)}$, denoted by $v_1,\dots,v_{\gamma^*}$.
Note that $u\in W_{i}\setminus S^{(2)}\subseteq V_{i}\setminus S^{(2)}$.
By the definition of $N_F(H,e)$, the families of copies counted by
$N_F(H,uv_j)$, $j\in[\gamma^*]$, are pairwise disjoint.
Then by Claim \ref{claim5.8C}, we have
\begin{align*}
N_F(H)
\geq \sum_{j\in[\gamma^*]}\!\!N_F(H,uv_j)\geq \gamma^* \gamma \cdot c(n,F)\geq  c(n,F),
\end{align*}
which contradicts Claim \ref{claim5.9C}.
Therefore, we conclude that $\cup_{i\in[r]}W_i\subseteq S^{(2)}$.
\end{proof}

\begin{claim}\label{claim5.11C}
Let $\bm{x}=(x_v)_{v\in V(H)}$ denote a non-negative unit eigenvector corresponding to $\rho(H)$,
with $x_{u^*}=\max_{v\in V(H)}x_v$.
Then, we have $x_u\geq (1\!-\!16f^2\varepsilon_{1})x_{u^*}$ for any $u\in \cup_{i\in [r]}V_i^{(1)}$.
\end{claim}

\begin{proof}
Fix $u\in V_i^{(1)}$.
Since $u\in V_i^{(1)}=V_i\setminus (W_i\cup S^{(1)})$,
the definitions of $W_i$ and $S^{(1)}$ imply that
$d_{V_i}(u)<4\varepsilon_1n$ and $d_H(u)>(1-\frac1r-4\varepsilon_1)n$.
Also, we know that $d_{V_i}(u)\leq |V_{i}|\leq(\frac1r+3\varepsilon_1)n$ for $i\in [r]$.
Thus, for any $j\neq i$, we obtain
\begin{align}\label{align-029F}
d_{V_{j}}(u)
&=d_H(u)\!-\!d_{V_{i}}(u)\!-\!\!\!\sum_{s\in [r]\setminus\{i,j\}}\!\!\!d_{V_{s}}(u)
>\big(1\!-\!\frac1r\!-\!8\varepsilon_1\big)n\!-\!\big(r\!-\!2\big)\big(\frac{1}{r}\!+\!3\varepsilon_1\big)n\nonumber\\
&=\big(\frac{1}{r}\!-\!(3r\!+\!2)\varepsilon_1\big)n\geq\big(\frac{1}{r}\!-\!3f\varepsilon_1\big)n.
\end{align}
Recall that $|V_j|\leq(\frac1r+3\varepsilon_{1})n$ for any $j\in [r]$.
Combining this with \eqref{align-029F}, we obtain
\begin{align}\label{align-030F}
\sum_{j\in [r]\setminus \{i\}}\!\!\big|\overline{N}_{V_{j}}(u)\big|
=\!\!\sum_{j\in [r]\setminus \{i\}}\!\!\big(|V_j|\!-\!d_{V_j}(u)\big)\leq 3(f\!+\!1)(r\!-\!1)\varepsilon_{1} n
<4(f^2\!-\!1)\varepsilon_{1} n,
\end{align}
where $\overline{N}_{V_{j}}(u):=V_j\setminus N_{V_j}(u)$.

Now, choose two vertices $u_1,u_2\in V_i^{(1)}$. We may assume that $x_{u_1}\geq x_{u_2}.$
Since $u_1\in V_i\setminus W_i$, we have $d_{V_i}(u_1)<4\varepsilon_{1} n$.
Note that $N_{V_i}(u_1)\setminus N_{V_i}(u_2)\subseteq N_{V_i}(u_1)$
and $N_{V_j}(u_1)\setminus N_{V_j}(u_2)\subseteq \overline{N}_{V_{j}}(u_2)$
for any $j\in [r]\setminus \{i\}$.
Combining this with \eqref{align-030F}, we can deduce that
\begin{align*}
\rho(H) (x_{u_1}\!-\!x_{u_2})
&\leq\!\!\!\sum_{v\in N_H(u_1)\setminus N_H(u_2)}\!\!\!x_v
\leq \Big(d_{V_i}(u_1)+\!\!\!\sum_{j\in [r]\setminus \{i\}}\!\!\!\big|\overline{N}_{V_{j}}(u_2)\big|\Big)x_{u^*}
\leq 4f^2\varepsilon_{1} nx_{u^*}.
\end{align*}

From Claim \ref{claim5.3C},
we know that
$\sqrt{2h/(1-\frac1r)}\geq n\big/(1+\varepsilon_{1})\geq(1-\varepsilon_{1})n$.
Combining this with Claim \ref{claim5.2C} yields that
\begin{align}\label{align-032F}
\rho(H)\geq \sqrt{\big(1\!-\!\frac{1}{r}\big)2h}\geq \big(1\!-\!\frac{1}{r}\big)\big(1\!-\!\varepsilon_{1}\big)n
\geq\big(1\!-\!\frac{1}{r}\!-\!\varepsilon_{1}\big)n.
\end{align}
This simplifies to $\rho(H)>\frac12n$.
Recall that $\rho(H) (x_{u_1}-x_{u_2})\leq 4f^2\varepsilon_{1} nx_{u^*}$. Thus, we have
\begin{align}\label{align-031F}
x_{u_1}-x_{u_2}\leq 8f^2\varepsilon_{1} x_{u^*}.
\end{align}
Note that $\rho(H) x_{u^*}=\sum_{u\in N_H(u^*)}x_u\leq d_H(u^*)x_{u^*}$.
Combining this with \eqref{align-032F} gives $d_H(u^*)\geq \rho(H)\geq (1-\frac{1}{r}-\varepsilon_{1})n.$
Hence, $u^*\notin S^{(1)}$.
By Claim \ref{claim5.10C}, we further conclude that $u^*\in \cup_{i\in [r]}V_i^{(1)}$.
Assume without loss of generality that $u^*\in V_1^{(1)}$.
In light of \eqref{align-031F}, we have $x_{u^*}-x_u\leq8f^2\varepsilon_{1} x_{u^*}$.
Thus, we have $x_u\geq (1-8f^2\varepsilon_{1})x_{u^*}$ for any $u\in V_1^{(1)}$.

It remains to prove that $x_u\geq (1-16f^2 \varepsilon_{1})x_{u^*}$ for any $u\in \cup_{i=2}^{r}V_i^{(1)}$.
Suppose to the contrary that there exists a vertex $u_0\in \cup_{i=2}^{r}V_i^{(1)}$ with $x_{u_0}<(1-16f^2\varepsilon_{1})x_{u^*}$.
Without loss of generality, assume that $u_0\in V_2^{(1)}$.
From \eqref{align-031F}, we know that $x_u-x_{u_0}\leq 8f^2\varepsilon_{1} x_{u^*}$ for $u\in V_2^{(1)}$ with $x_u\geq x_{u_0}$.
This further implies that $x_u\leq (1-8f^2\varepsilon_{1})x_{u^*}$ for all $u\in V_2^{(1)}$.
By Claims \ref{claim5.4C}, \ref{claim5.5C}, and \ref{claim5.6C},
we have $|V_i|\leq(\frac1r+3\varepsilon_{1})n$ and $|W_i\cup S^{(1)}|\leq2\varepsilon_{1}n$ for each $i\in[r]$.
Recall that $V_2^{(1)}=V_2\setminus (W_2\cup S^{(1)})$. It follows that
\begin{align*}
\sum_{v\in N_{V_2}(u^*)}x_v
&\leq\sum_{v\in V_2^{(1)}}x_v\!+\!\!\sum_{v\in V_2\setminus V_2^{(1)}}\!\!x_v
\leq\Big(\big(1-8f^2\varepsilon_{1}\big)\big|V_2^{(1)}\big|+\big|W_2\cup S^{(1)}\big|\Big)x_{u^*}\\
&\leq\Big(\big(1-8r^2\varepsilon_{1}\big)\big(\frac{1}{r}\!+\!3\varepsilon_{1}\big)n+2\varepsilon_{1}n\Big)x_{u^*}.
\end{align*}
Since $u^*\in V_1^{(1)}\subseteq V_1\setminus W_1$, we have
$d_{V_1}(u^*)<4\varepsilon_{1} n$ and $\sum_{v\in N_{V_1}(u^*)}x_v\leq d_{V_1}(u^*)x_{u^*}\le4\varepsilon_{1} nx_{u^*}.$
Moreover, for every $i\in [r]\setminus\{1,2\}$, we have
$\sum_{v\in N_{V_i}(u^*)}x_v\leq|V_i|x_{u^*}\le (\frac1r+3\varepsilon_{1})nx_{u^*}.$
Combining the above inequalities derived from three cases, we deduce that
$$\rho(H)x_{u^*}=\sum_{i\in[r]}\sum_{v\in N_{V_i}(u^*)}x_v
\leq \Big(\big(r-1-8r^2\varepsilon_{1}\big)\big(\frac{1}{r}\!+\!3\varepsilon_{1}\big)n+6\varepsilon_{1}n\Big)x_{u^*}.$$
This simplifies to $\rho(H)\leq (1-\frac1r-4r\varepsilon_{1})n$,
contradicting \eqref{align-032F}.
Thus, the proof is complete.
\end{proof}

\begin{claim}\label{claim5.12C}
We have $2hx_u^2\leq(1-2\varepsilon_{2})d_H(u)$ for any $u\in S^{(2)}$.
\end{claim}

\begin{proof}
Recall from Claim \ref{claim5.3C} that $(1\!-\!\varepsilon_{1})\psi\leq n\leq (1\!+\!\varepsilon_{1})\psi$.
Furthermore, by Claims \ref{claim5.10C} and \ref{claim5.5C},
we have $\sum_{i\in[r]}|V_i^{(1)}|=n\!-\!|S^{(1)}|\geq (1\!-\!\varepsilon_{1})n\geq(1\!-\!\varepsilon_{1})^2\psi.$
This, together with Claim \ref{claim5.11C}, gives
$$1=\!\sum_{v\in V(H)}\!\!x_v^2 \geq
\sum_{i\in[r]}\sum_{v\in V_i^{(1)}}\!x_v^2
\geq\big(1\!-\!\varepsilon_{1}\big)^2\psi\big(1\!-\!16f^2\varepsilon_{1}\big)^2x^2_{u^*}
\geq\big(1\!-\!\frac12\varepsilon_{2}\big)\psi x_{u^*}^2,$$
where the last inequality holds as $\varepsilon_{1}\ll \varepsilon_{2}$.
This can be further simplified to
\begin{align}\label{align-031J}
\psi x_{u^*}^2\leq\frac{1}{1\!-\!\frac12\varepsilon_{2}}\leq 1+\varepsilon_{2}.
\end{align}
By Claim \ref{claim5.2C}, we have $\rho(H)\geq\sqrt{(1\!-\!\frac{1}{r})2h}=(1\!-\!\frac1r)\psi.$
Combining this with \eqref{align-031J} yields
\begin{align}\label{align-031K}
\big(1\!-\!\frac1r\big)\psi x_u\leq \rho(H)x_u=\!\sum_{v\in N_H(u)}\!\!x_v\leq d_H(u)x_{u^*}
\leq d_H(u)\sqrt{\big(1\!+\!\varepsilon_{2}\big)/\psi}.
\end{align}
By Claim \ref{claim5.3C}, we have
$n\leq(1+\varepsilon_1)\psi$.
Thus, for any $u\in S^{(2)}$,
the definition of $S^{(2)}$ yields
\begin{align}\label{align-031L}
d_H(u)\leq \big(1\!-\!\frac{1}{r}\!-\!4\varepsilon_{2}\big)n\leq \big(1\!-\!\frac{1}{r}\!-\!4\varepsilon_{2}\big)\big(1\!+\!\varepsilon_{1}\big)\psi
\leq \big(1\!-\!\frac{1}{r}\!-\!3\varepsilon_{2}\big)\psi.
\end{align}
Note that $2h=(1\!-\!\frac1r)\psi^2.$
Combining \eqref{align-031K} and \eqref{align-031L}, we can deduce that
$$\frac{2hx_u^2}{d_H(u)}=\big(1\!-\!\frac1r\big)\frac{\psi^2x_u^2}{d_H(u)}
\leq d_H(u)\frac{1\!+\!\varepsilon_{2}}{\psi(1\!-\!\frac1r)}
\leq \frac{(1\!+\!\varepsilon_{2})(1\!-\!\frac1r\!-\!3\varepsilon_{2})}{1\!-\!\frac1r}
\leq(1\!+\!\varepsilon_{2})(1\!-\!3\varepsilon_{2}).$$
This can be further simplified to
$2hx_u^2\leq (1-2\varepsilon_{2})d_H(u)$,
thus proving the claim.
\end{proof}

\begin{claim}\label{claim5.13C}
We have $S^{(2)}=\varnothing$.
\end{claim}

\begin{proof}
Suppose, for the sake of contradiction, that there exists a vertex $u_0\in S^{(2)}$.
Recall that $\frac12m<h\leq m$ and that $\psi^2=2h/(1\!-\!\frac1r)$.
Then, we have $\psi=\Theta(\sqrt{m})$, which is sufficiently large.
Since $\varepsilon_{2}$ is a constant with $0<\varepsilon_{2}\ll1$,
it follows from \eqref{align-031J} that
$x^2_{u_0}\leq x^2_{u^*}\leq \frac{1}{\psi}(1+\varepsilon_{2})\leq \frac12\varepsilon_{2}.$
This further implies that $1-x^2_{u_0}\geq1-\frac12\varepsilon_{2}\geq 1\big/(1+\varepsilon_{2}).$

By Claim \ref{claim5.2B}, the graph $H$ contains no light edges. Moreover, since $r\geq 3$, Claim \ref{claim5.2C} gives
$\rho^2(H)\geq(1-\frac1r)2h\geq \frac{4}{3}h$.
Therefore, we may apply part (iii) of Lemma \ref{lem2.8C} to $H$ with $a=\frac{4}{3}$. It follows that
\[
\delta(H)\geq \bigl(\frac{a-1}{8}\bigr)^9\sqrt{h}
=\bigl(\frac{1}{24}\bigr)^9\sqrt{h}
=\frac{\sqrt{h}}{24^9}.
\]
Since $h$ is sufficiently large, we have $\delta(H)\geq 2$.
Hence, deleting $u_0$ from $H$ creates no isolated vertices.
Now, let $H_0$ be the graph obtained from $H$ by deleting
the vertex $u_0$ and all edges incident with $u_0$. Then
$V(H_0)=V(H)\setminus\{u_0\}$.
Consider the vector
$\bm{x}'=\bm{x}|_{V(H_0)}$, which is the restriction of $\bm{x}$ to
$V(H_0)$. Therefore,
${\bm{x}'}^T\bm{x}'=1-x_{u_0}^2$.
By the Rayleigh quotient,
$$\rho(H_0)
\geq \frac{{\bm{x}'}^T A(H_0)\bm{x}'}{{\bm{x}'}^T\bm{x}'}
=\frac{1}{1-x_{u_0}^2}\displaystyle\sum_{uv\in E(H_0)}\!\!2x_ux_v.$$
Combining this with $\sum_{u\in N_H(u_0)}x_u=\rho(H)x_{u_0}$, we can deduce that
\begin{align*}
\rho(H)=\bm{x}^TA(H)\bm{x}=\!\!\!\sum_{uv\in E(H_0)}\!\!\!\!2x_ux_v+2x_{u_0}\!\!\!\!\sum_{u\in N_H(u_0)}\!\!\!\!x_u
\leq\rho(H_0)(1\!-\!x_{u_0}^2)\!+\!2\rho(H)x_{u_0}^2.
\end{align*}
Rearranging gives $(1-x^2_{u_0})\rho(H_0)\geq (1-2x^2_{u_0})\rho(H)$.
This, together with $1-x^2_{u_0}\geq \frac1{1+\varepsilon_{2}}$, yields
$$\rho(H_0)\geq\rho(H)\Big(1\!-\!\frac{x^2_{u_0}}{1\!-\!x^2_{u_0}}\Big)
\geq\rho(H)\Big(1\!-\!\big(1\!+\!\varepsilon_{2}\big)x^2_{u_0}\Big).$$
Recall that $\Phi(H)=\rho(H)/\sqrt{e(H)}=\rho(H)/\sqrt{h}$. Thus, we have
\begin{align}\label{align-033F}
\Phi(H_0)\!-\!\Phi(H)=\frac{\rho(H_0)\sqrt{h}\!-\!\rho(H)\sqrt{e(H_0)}}{\sqrt{h}\sqrt{e(H_0)}}
\geq \frac{\rho(H)\Big(\big(1\!-\!(1\!+\!\varepsilon_{2})x^2_{u_0}\big)\sqrt{h}\!-\!\sqrt{e(H_0)}\Big)}
{\sqrt{h}\sqrt{e(H_0)}}.
\end{align}
Note that $\sqrt{h}\!-\!\sqrt{e(H_0)}=\frac{d_H(u_0)}{\sqrt{h}+\sqrt{e(H_0)}}
\geq \frac{d_H(u_0)}{2\sqrt{h}}$, and by Claim \ref{claim5.12C},
$(1+\varepsilon_{2})\frac{2hx_{u_0}^2}{d_H(u_0)}
\leq (1+\varepsilon_{2})(1-2\varepsilon_{2})\leq 1-\varepsilon_{2}.$
It follows that
\begin{align*}
\big(1\!-\!(1\!+\!\varepsilon_{2})x^2_{u_0}\big)\sqrt{h}\!-\!\sqrt{e(H_0)}
\geq\frac{d_H(u_0)}{2\sqrt{h}}\Big(1\!-\!(1+\varepsilon_{2})\frac{2hx_{u_0}^2}{d_H(u_0)}\Big)
\geq\frac{\varepsilon_{2}d_H(u_0)}{2\sqrt{h}}.
\end{align*}

From Claim \ref{claim5.2C},
we know that $\rho(H)\geq\sqrt{(1\!-\!\frac1r)2h}\geq\sqrt{h}\geq\sqrt{e(H_0)}$.
Substituting the above two inequalities into \eqref{align-033F},
we derive that
$$\Phi(H_0)\!-\!\Phi(H)\geq\frac{\varepsilon_{2}d_H(u_0)}{2h}\geq\frac{\varepsilon_{0}d_H(u_0)}{h}.$$
Moreover, since $u_0\in S^{(2)}$, it follows from \eqref{align-031L} that
$$e(H)\!-\!e(H_0)
=d_H(u_0)\leq\big(1\!-\!\frac{1}{r}\!-\!3\varepsilon_{2}\big)\psi\leq\psi=\sqrt{{2h}\big/{\big(1\!-\!\frac{1}{r}\big)}}.$$
By Definition \ref{def5.1}, $H_0$ is an $\varepsilon_0$-dense
subgraph of $H$. This contradicts Claim \ref{claim5.1C}, since
$H=G_{(\ell)}$ is an $\varepsilon_0$-core.
Thus, we have $S^{(2)}=\varnothing$, as claimed.
\end{proof}

Claims \ref{claim5.10C} and \ref{claim5.13C} now give
$\cup_{i\in[r]}W_i\subseteq S^{(2)}=\varnothing$.
Consequently, $V_i^{(2)}=V_i$ for every $i\in[r]$.
A small number of internal edges
may still be present, and the purpose of the next step is to exclude
them by a sharp supersaturation argument.

\medskip

\textbf{Step 4: Exclusion of internal edges by supersaturation.}
We now combine the almost Tur\'{a}n structure obtained in Step 2 with the elimination of exceptional vertices established in Step 3 to exclude all internal edges.
The key common-neighborhood property shows that any internal edge can be extended to many copies of $F$.

Since $\cup_{i\in[r]}W_i\subseteq S^{(2)}=\varnothing$,
we have $V_i^{(2)}=V_i\setminus(W_i\cup S^{(2)})=V_i$ for any $i\in [r]$.
Recall that $\cup_{i\in[r]}V_i$ is a partition of $V(H)$,
and our aim is to show that $H\subseteq K_{V_1,V_2,\dots,V_r}$.
Suppose to the contrary that there exists an edge $uv$ in $\cup_{i\in [r]}E(H[V_i])$.
Without loss of generality, we may assume that $uv\in E(H[V_1])$.

Since $F$ is color-critical with $\chi(F)=r+1$, it has at least
one good edge. Fix an arbitrary good edge $xy\in E(F)$ and an
arbitrary proper $r$-coloring $\chi_{xy}$ of $F-xy$ in which both
$x$ and $y$ receive color $1$.
Recall that
$\tau_{xy}^{(i)}:=\big|\{z\in V(F)\setminus\{x,y\}:\chi_{xy}(z)=i\}\big|$
for each $i\in[r]$.
Let $\operatorname{Cont}_H(\chi_{xy})$ denote the number of
edge-preserving injections $\varphi:V(F)\to V(H)$ such that
$\{\varphi(x),\varphi(y)\}=\{u,v\}$
and $\varphi(z)\in V_i$ whenever $\chi_{xy}(z)=i$.
Following the discussion leading to \eqref{align-00G}, we estimate
the contribution $\operatorname{Cont}_H(\chi_{xy})$.

\begin{enumerate}
\item[\rm(i)]
First, there are two ways to map $\{x,y\}$ onto $\{u,v\}$.
By Claim~\ref{claim5.4C} and the hierarchy
$\varepsilon_1\ll\varepsilon_2$, for sufficiently large $n$ we have
\[
|V_1\setminus\{u,v\}|
=|V_1|-2
\geq\Big(\frac1r-3\varepsilon_1\Big)n-2
\geq\Big(\frac1r-\varepsilon_2\Big)n.
\]
Hence, the remaining $\tau_{xy}^{(1)}$ vertices of color $1$
can be mapped injectively into $V_1\setminus\{u,v\}$ in at least
$\big((\frac1r-\varepsilon_2)n\big)_{\tau_{xy}^{(1)}}$ ways,
using the falling factorial notation in \eqref{align-00H}.
The vertices chosen from $H$ form a $(\tau_{xy}^{(1)}\!+\!2)$-set $\widetilde{L}_1$.

\item[\rm(ii)]
Since $|\widetilde{L}_1|=2+\tau_{xy}^{(1)}\leq f$ and
$\widetilde{L}_1\subseteq V_1^{(2)}=V_1$, Claim~\ref{claim5.7C}(i),
with $k=2$ and $i_0=2$, shows that the vertices of
$\widetilde{L}_1$ have at least
$(\frac1r-3f^2\varepsilon_2)n$ common neighbors in $V_2$.
Thus, the $\tau_{xy}^{(2)}$ vertices of color $2$ can be mapped
injectively into this common neighborhood in at least
$\big((\frac1r-3f^2\varepsilon_2)n\big)_{\tau_{xy}^{(2)}}$
ways. Denote their image by $\widetilde{L}_2$.

\item[\rm(iii)]
For $j=3,\ldots,r$, suppose that
$\widetilde{L}_1,\ldots,\widetilde{L}_{j-1}$ have been chosen.
Then
\[
1\leq \Big|\bigcup_{i=1}^{j-1}\widetilde{L}_i\Big|
=2+\sum_{i=1}^{j-1}\tau_{xy}^{(i)}\leq f,
\qquad
\bigcup_{i=1}^{j-1}\widetilde{L}_i
\subseteq\bigcup_{i\neq j}V_i^{(2)}.
\]
By Claim~\ref{claim5.7C}(i), with $k=2$ and $i_0=j$,
these vertices have at least
$(\frac1r-3f^2\varepsilon_2)n$ common neighbors in $V_j$.
Consequently, the $\tau_{xy}^{(j)}$ vertices of color $j$
can be mapped into this common neighborhood in at least
$\big((\frac1r-3f^2\varepsilon_2)n\big)_{\tau_{xy}^{(j)}}$
ways.
\end{enumerate}
Since every edge of $F-xy$ joins distinct color classes and
$xy$ is mapped to $uv$, every injection constructed above is
edge-preserving.

Consequently, we can deduce that
\begin{align*}
\operatorname{Cont}_H(\chi_{xy})
\geq
2\Big(\big(\frac1r-\varepsilon_2\big)n\Big)_
{\tau_{xy}^{(1)}}
\prod_{i=2}^{r}
\Big(\big(\frac1r-3f^2\varepsilon_2\big)n\Big)_
{\tau_{xy}^{(i)}}.
\end{align*}

Now let
$T_{n,r}^*=K_{U_1,\ldots,U_r}+u'v'$
be a graph attaining $c(n,F)$, where the added edge $u'v'$ lies
inside $U_1$, and set $t_i:=|U_i|$ for each $i\in[r]$.
Since $T_{n,r}$ is balanced, we have
$t_i\in
\bigl\{\lfloor\frac nr\rfloor,\lceil\frac nr\rceil\bigr\}$,
and hence $t_i=\frac nr+O(1)$ for every $i\in[r]$.
Set $D_n:=6rf^2\varepsilon_2+2r/n$.
For sufficiently large $n$, all factors in the falling factorials
$(t_1-2)_{\tau_{xy}^{(1)}}$ and $(t_i)_{\tau_{xy}^{(i)}}$
are at least $n/(2r)$.
Since $|t_i-n/r|\leq1$, the ratio of each corresponding factor
in the preceding lower bound to that in
$2(t_1-2)_{\tau_{xy}^{(1)}}\prod_{i=2}^r(t_i)_{\tau_{xy}^{(i)}}$
is at least $1-D_n$.
As $\sum_{i=1}^r\tau_{xy}^{(i)}=f-2$, this yields, uniformly
over all good edges $xy$ and their proper $r$-colorings,
\begin{align*}
&2\Big(\big(\frac1r-\varepsilon_2\big)n\Big)_{\tau_{xy}^{(1)}}
\prod_{i=2}^{r}
\Big(\big(\frac1r-3f^2\varepsilon_2\big)n\Big)_{\tau_{xy}^{(i)}}\\
&\qquad\geq (1-D_n)^{f-2}
2(t_1-2)_{\tau_{xy}^{(1)}}
\prod_{i=2}^{r}(t_i)_{\tau_{xy}^{(i)}}\\
&\qquad\geq \bigl(1-(f-2)D_n\bigr)
2(t_1-2)_{\tau_{xy}^{(1)}}
\prod_{i=2}^{r}(t_i)_{\tau_{xy}^{(i)}}.
\end{align*}
Recall that $\varepsilon_2\ll\varepsilon_3$ and $n$ is sufficiently large.
It follows that
\[
\operatorname{Cont}_H(\chi_{xy})
>
(1-\varepsilon_3)
2(t_1-2)_{\tau_{xy}^{(1)}}
\prod_{i=2}^{r}(t_i)_{\tau_{xy}^{(i)}}.
\]

It remains to justify that these contributions can be summed
without overcounting. Let $\mathcal I_{uv}$ be the set of
edge-preserving injections $\varphi:V(F)\to V(H)$ for which $uv$
is the unique edge of $\varphi(F)$ lying inside a partition class.
Each $\varphi\in\mathcal I_{uv}$ determines a unique edge
$xy\in E(F)$ mapped onto $uv$. The partition then induces a
proper $r$-coloring $\chi_\varphi$ of $F-xy$ with
$\chi_\varphi(x)=\chi_\varphi(y)=1$.
Since $\chi(F)=r+1$ and deleting one edge decreases the
chromatic number by at most one, $\chi(F-xy)=r$; hence $xy$
is a good edge. Conversely, every injection counted by
$\operatorname{Cont}_H(\chi_{xy})$ has $uv$ as its unique
internal edge. Thus the injections counted by the different
pairs $(xy,\chi_{xy})$ form a disjoint partition of
$\mathcal I_{uv}$. By the definition of $N_F(H,uv)$,
\[
|\operatorname{Aut}(F)|\,N_F(H,uv)
=\sum_{\substack{xy\in E(F)\\xy\text{ is good}}}
\ \sum_{\substack{\chi_{xy}\text{ proper}\\
\chi_{xy}(x)=\chi_{xy}(y)=1}}
\operatorname{Cont}_H(\chi_{xy}).
\]
The preceding termwise bound, together with \eqref{align-01F},
therefore gives
\[
N_F(H)\geq N_F(H,uv)>(1-\varepsilon_3)c(n,F),
\]
contradicting Claim~\ref{claim5.9C}.
Thus $H\subseteq K_{V_1,V_2,\dots,V_r}$.

\medskip

\textbf{Step 5: Identification of the extremal structure.}
Upon completing Step~4, we conclude that $H\subseteq K_{V_1,V_2,\ldots,V_r}$,
and thus $H$ is $K_{r+1}$-free. By Theorem~\ref{thm1.7C}, we have $\rho^2(H)\leq (1-\frac{1}{r})2h$,
with equality if and only if $H$ is regular and complete \(r\)-partite.
Combining this with Claim \ref{claim5.2C} gives that $\ell=1$
and $\rho^2(H)=(1-\frac{1}{r})2h$.
It follows that $H=G_{(1)}=G$ and $H$ is regular and complete \(r\)-partite,
which contradicts the assumption that $G$ is not a regular complete $r$-partite graph.

This completes the proof of Theorem \ref{thm1.8C}.
\end{proof}

\section{Concluding remarks}

Nikiforov's light-edge deletion method
\cite{Nikiforov2009F,Nikiforov2021Nosal}
removes edges with small Perron-coordinate products,
thereby controlling the loss in spectral radius.
This method has also been used in
\cite{NZ2021b,Li2025+}.
However, the absence of light edges does not
guarantee that deleting several edges simultaneously
cannot increase the normalized spectral radius.
The $\varepsilon$-core condition rules out such
deletions when the number of removed edges is
within the prescribed range and the increase is
sufficiently large. Together with Perron-vector
and common-neighborhood estimates, this condition
allows us to exclude vertices with small total
degree or large internal degree.
The following example shows why the absence
of light edges alone is insufficient.

Let $k$ be a sufficiently large even integer, and let $G$
be obtained from $K_{k,k}$ by deleting exactly $k/2$ edges incident
with one vertex. A direct calculation shows that
$\min_{uv\in E(G)}x_ux_v=\frac{1+o(1)}{4k}>\frac{1}{8\sqrt{e(G)}}$,
so $G$ has no light edges. On the other hand, deleting that vertex
leaves $H=K_{k-1,k}$, and
\[
\Phi(H)-\Phi(G)
=\big(\frac14+o(1)\big)
\frac{e(G)-e(H)}{e(G)}.
\]
Hence, for every fixed $0<\varepsilon<1/5$ and sufficiently large $k$,
the graph $H$ is an $\varepsilon$-dense subgraph of $G$.
Thus, a graph with no light edges need not be an $\varepsilon$-core.

Lastly, Theorem~\ref{thm1.8C} addresses the critical spectral regime.
A key step in the proof of Theorem~\ref{thm1.8C} is to reduce a counterexample to an
$\varepsilon_0$-core while retaining almost all its edges.
By edge-spectral stability,
the resulting graph is close to a Tur\'an graph.
We then use Perron-vector estimates to exclude
vertices with small total degree or large internal
degree. Finally, common-neighborhood estimates
show that any remaining internal edge would
produce too many copies of $F$, yielding
the desired contradiction.
We also conjectured that, for every fixed $C>0$,
any $m$-edge graph $G$ satisfying
$\rho(G)\geq\sqrt{2(1-1/r)m}+C$
must contain $\Omega(m^{(f-1)/2})$ copies of $F$.
This conjecture was subsequently confirmed
by Chen and Li~\cite{ChenLi2026}, who also
determined the asymptotically sharp coefficient
for sufficiently small fixed $C>0$.

Several related questions remain open. It would be interesting to
obtain analogous sharp results for graphs $F$ with $\chi(F)=3$, for
which the spectral extremal configuration may depend more sensitively
on the structure of $F$. Another natural direction is to characterize
how the minimum number of copies of $F$ depends on the excess of the spectral radius above the critical value.
More generally, it would be worthwhile to investigate whether the \(\varepsilon\)-core method can be applied to
spectral supersaturation problems whose extremal graphs are not of Tur\'{a}n type.

%

\end{document}